\documentclass{amsart}[12pt]

\usepackage{amsmath,amsfonts,amssymb,latexsym,amsthm}
\usepackage{graphicx,epsf}
\usepackage{epsfig}
\usepackage{comment}
\usepackage{subcaption}
\usepackage{amssymb}
\usepackage{amsthm}
\usepackage{multicol}
\usepackage{mathtools}
\usepackage{tikz}
\usetikzlibrary{automata,positioning,arrows}
\usepackage{setspace}
\usepackage{psfrag}
\usepackage{tikz}
\usepackage[T1]{fontenc}
\usetikzlibrary{arrows,matrix,positioning}

\newlength{\ml}
\newtheorem{thm}{Theorem}[section]
\newtheorem{lem}[thm]{Lemma}

\newtheorem{cor}[thm]{Corollary}
\newtheorem{prop}[thm]{Proposition}
\newtheorem{conj}[thm]{Conjecture}
\newtheorem{open}[thm]{Open Problem}

\def\bbs{{\text{\small $\boxed{B}$\,}}}

\def\zz{\mathbb Z}
\def\nn{\mathbb N}

\def\sm{\smallsetminus}

\def\Ga{\Gamma}

\def\la{\lambda}

\def\si{\sigma}

\def\al{\alpha}

\def\om{\omega}

\def\cC{\mathcal C}
\def\ca{\mathcal A}
\def\cd{\mathcal D}

\def\cf{\mathcal F}

\def\cp{\mathcal P}

\def\cw{\mathcal W}

\def\br{\mathbf{r}}
\def\ssu{\subset}

\def\ssv{\subseteq}

\def\<{\langle}
\def\>{\rangle}

\def\oa{{\overline{\al}}}

\def\ts{\hskip.015cm}

\def\0{{\mathbf 0}}

\DeclareSymbolFont{extraup}{U}{zavm}{m}{n}
\DeclareMathSymbol{\vardiamond}{\mathalpha}{extraup}{87}

\newcommand{\lra}{\leftrightarrow}

\def\rA{\textrm{A}}

\def\.{\hskip.06cm}
\def\ts{\hskip.03cm}
\def\mts{\hspace{-.04cm}}

\def\ella{\rho}
\def\kK{m}

\def\nin{\noindent}
\def\ns{\normalsize}

\def\P{{{\rm{\textsf{P}}}}}

\def\ParP{{\rm{\textsf{$\oplus$P}}}}

\def\PH{{\rm{\textsf{PH}}}}
\def\BPP{{\rm{\textsf{BPP}}}}
\def\SP{{{\rm{\textsf{\#P}}}}}

\def\NP{{{\rm{\textsf{NP}}}}}
\def\EXP{{{\rm{\textsf{EXP}}}}}
\def\NEXP{{{\rm{\textsf{NEXP}}}}}
\def\SEXP{{{\rm{\textsf{\#EXP}}}}}
\def\EEXP{{{\rm{\textsf{EXPSPACE}}}}}

\def\PAR{{\rm{\textsf{$\oplus$EXP}}}}

\def\ord{\eta}

\begin{document}
\title{Pattern Avoidance Is Not P-recursive}

\author[Scott Garrabrant \and Igor Pak]{Scott~Garrabrant$^{\star }$ \and Igor~Pak$^{\star }$}

\thanks{\thinspace~${\hspace{-.45ex}}^\star $Department of Mathematics,
UCLA, Los Angeles, CA, 90095.
\hskip.06cm
Email:
\hskip.06cm
\texttt{\{coscott,\ts\/pak\}@math.ucla.edu}}

\date{\today}


\begin{abstract}
Let $\cf \subset S_k$ be a finite set of permutations and
let  $C_n(\mathcal{F})$  denote the number of permutations
$\si\in S_n$ \emph{avoiding} the set of patterns~$\cf$.
The \emph{Noonan--Zeilberger conjecture} states that the sequence
$\{C_n(\mathcal{F})\}$ is P-recursive.  We disprove this
conjecture.

The proof uses Computability Theory and builds on our
earlier work~\cite{GP1}.  We also give two applications
of our approach to complexity of computing $\{C_n(\mathcal{F})\}$.
\end{abstract}

\maketitle

\vskip.9cm

\section{Introduction}

\nin
Combinatorial sequences have been studied for centuries,
with results ranging from minute properties of individual
sequences to broad results on large classes of sequences.
Even just listing the tools and ideas can be exhausting,
which range from algebraic to bijective, to probabilistic
and number theoretic~\cite{Rio}.  The existing technology
is so strong, it is rare for an open problem to remain
unresolved for more than a few years, which makes the
surviving conjectures all the more interesting and
exciting.

The celebrated \emph{Noonan--Zeilberger Conjecture} is one such open problem.
It is in the area of \emph{pattern avoidance} which has been
very popular in the past few decades, see e.g.~\cite{Bona,Kit}.
The problem was first raised as a question by Gessel in 1990,
see~\cite[$\S$10]{Ges}.
In 1996, it was upgraded to a conjecture and further investigated by
Noonan and Zeilberger~\cite{NZ}, see also~$\S$\ref{ss:fin-rem-ADE}.
There are now hundreds of papers
in the area with positive results for special sets of patterns, all
in support of the conjecture, and only two recent experimental
clouds against it (see~$\S$\ref{ss:fin-rem-1324}).
Here we disprove the conjecture.

\medskip

Let $\si \in S_n$ and $\om \in S_k$.  Permutation $\si$ is said
to \emph{contain} the \emph{pattern} $\om$ if there is a subset
$X\subseteq \{1,\ldots,n\}$, $|X|=k$, such that $\si|_X$ has
the same relative order as~$\om$.  Otherwise, $\si$
is said to \emph{avoid}~$\om$.
Fix a set of patterns $\cf \subset S_k$. Denote by $C_n(\mathcal{F})$
the number of permutations $\si\in S_n$ \emph{avoiding}
the patterns $\om \in\cf$.  The sequence $\{C_n(\cf)\}$
is the main object in the area, extensively analyzed
from analytic, asymptotic and combinatorial points of view
(see $\S$\ref{ss:fin-rem-growth}).

An integer sequence $\{a_n\}$ is called \emph{polynomially recursive},
or \emph{P-recursive},
if it satisfies a nontrivial linear recurrence relation of the form
$$
q_0(n)\ts a_n \. + \.  q_1(n)\ts a_{n-1} \. + \. \ldots \. + \. q_k(n)\ts a_{n-k} \. = \.0\.,
$$
for some $q_i(x) \in \zz[x]$, $0\le i \le k$.  The study of P-recursive sequences
plays a major role in modern Enumerative and Asymptotic Combinatorics (see e.g.~\cite{FS,Odl,Sta}).
They have \emph{D-finite} (also called \emph{holonomic})
generating series and various asymptotic properties (see~$\S$\ref{ss:fin-rem-asy}).

\begin{conj}[Noonan--Zeilberger]\label{conj:NZT}
Let $\cf \ssu S_k$ be a fixed set of patterns.  Then the sequence
$\{C_n(\mathcal{F})\}$ is P-recursive.
\end{conj}

The following is the main result of the paper.

\begin{thm}\label{t:NZT}
The Noonan--Zeilberger conjecture is false. More precisely, there exists a set
of patterns $\mathcal{F} \ssu S_{80}$,  such that the sequence
$\{C_n(\mathcal{F})\}$ is \ts \emph{not} \ts P-recursive.
\end{thm}

We should mention that in contrast with most literature in the area which studies
small sets of patterns, our set~$\cf$ is enormously large and we make no
effort to decrease its size (cf.~$\S$\ref{ss:fin-rem-size-small}).
To be precise, we construct two large sets
$\cf,\cf' \ssu S_{80}$ and show that at least one of them gives
a counterexample to the conjecture.  In fact, it is conceivable
that a single permutation pattern $\om=(1324)$ may be sufficient
(see~$\S$\ref{ss:fin-rem-1324}).

\smallskip

The proof of Theorem~\ref{t:NZT} is based on the following idea.
Roughly, we show that every \emph{two-stack automaton}~$M$
can be emulated by a finite set of permutation patterns.
More precisely, we show that the number of accepted paths
of~$M$ is equal to $C_n(\cf)$~mod~2, for a subset of
integers~$n$ forming an arithmetic progression
(see Main Lemma~\ref{l:mainL}).  This highly technical construction
occupies much of the paper.  The rest of the proof is based on our
approach in~\cite{GP1}, where we resolved \emph{Kontsevich's problem} on
the P-recursiveness of certain numbers of words in linear groups.
The ability to emulate \emph{any} two-stack automaton~$M$
in the weak sense described above is surprisingly powerful (see below).

\smallskip

Let us now present two applications of our results to the
complexity of computing $\{C_n(\cf)\}$.
Two sets of patterns $\cf_1$ and $\cf_2$ are called
\emph{Wilf--equivalent}, denoted $\cf_1\sim \cf_2$, if
$C_n(\cf_1) = C_n(\cf_2)$ for all~$n$.  In~\cite{Vat2},
Vatter asked whether it is decidable when two patterns
are Wilf--equivalent.  Here we resolve a mod-2 version of this
problem (cf.~Section~\ref{s:undecide} and
$\S$\ref{ss:fin-rem-P-recursive}).

\begin{thm}\label{t:undecide}
The problem whether \. $C_n(\cf_1) = C_n(\cf_2)$~{\rm mod}~$2$ \.
for all \. $n\in \nn$, is undecidable.
\end{thm}

For the second application, we say that an integer sequence $\{a_n\}$ has
a \emph{Wilfian formula}, if it can be computed in time polynomial in~$n$.
This notion was introduced by Wilf in a different context~\cite{Wilf-answer}.
Recently, both Klazar~\cite{Kla} and Vatter~\cite{Vat2} asked whether every
sequence $\{C_n(\cf)\}$ has a Wilfian formula.

\begin{thm}\label{t:wilfian}
If \. {\rm $\EXP \ne \PAR$}, then there exists a finite set of patterns~$\cf$, such that
the sequence $\{C_n(\cf)\}$ cannot be computed in time polynomial in~$n$.
\end{thm}

This also seems to give a good answer to Wilf's original question
``Can one describe a reasonable and natural family of combinatorial enumeration
problems for which there is provably no polynomial-in-$n$ time formula or
algorithm to compute $f(n)$?''~\cite{Wilf-answer}
(see also~$\S$\ref{ss:fin-rem-wilfian}).

Here $\EXP$ is the complexity class of problems decidable
in exponential time, and $\PAR$ is the exponential time
analogue of the $\ParP$, so $\EXP \subseteq \PAR \ssv \EEXP$.
We refer to Section~\ref{s:wilfian} for the definitions,
details and references.
It is believed that $\EXP \ne \PAR$ (see~$\S$\ref{ss:wilfian-complexity}),
so the answer to the question by Klazar and Vatter is
likely negative (cf.~$\S$\ref{ss:fin-rem-complexity}).

\medskip

The rest of the paper is structured as follows.  We begin with an explicit
construction of a two-stack automata with a non-P-recursive number
of accepted paths (Section~\ref{s:2SA}).  In Section~\ref{s:proof-main-thm},
we reduce the proof of Theorem~\ref{t:NZT} to the Main Lemma~\ref{l:mainL}
on embedding two-stack automata into pattern avoidance problems.
The proof of the Main Lemma spans the next four sections.  We first
present the construction in Section~\ref{s:contruction}. In Section~\ref{s:MLP}, we prove
the lemma modulo a number of technical results.  We illustrate the
construction in a lengthy example in Section~\ref{s:example}, and
prove the technical results in Section~\ref{s:TL}.  We proceed to
prove theorems~\ref{t:undecide} and~\ref{t:wilfian} in sections~\ref{s:undecide}
and~\ref{s:wilfian}, respectively.  We conclude with final remarks and
open problems in Section~\ref{s:fin-rem}.

\bigskip

\section{Two-stack automata}\label{s:2SA}

\nin
In this section we construct an automaton with a non-P-recursive number
of accepted paths.  The construction is technical, but elementary.
Although it is more natural if the reader is familiar
with basic Automata Theory (see e.g.~\cite{HMU,Sip}),
the construction is completely self-contained and is given
in the language of elementary Graph Theory.  The proof, however,
is not self-contained and follows a similar proof in~\cite{GP1}.

To be precise, we give an explicit construction of a graph, where the
vertices have certain variables as weights.  We then count the number
$a_n$ of paths of length $n$ between two fixed vertices, where only
certain weight sequences are allowed (we call these \emph{balanced paths}).
The non-P-recursiveness of \ts $a_n~\text{mod}~2$ \ts
is explained below.

\smallskip

\subsection{The motivation}\label{ss:2SA-motivation}
It is relatively easy to present a construction of an automaton
which produces a non-P-recursive sequence $\{a_n\}$ of balanced paths.
Our goal in this section is stronger -- the sequence $\{a_n\}$ we
get is not equal to \emph{any} P-recursive sequence modulo~2.
Somewhat informally, we call such automaton \emph{non-P-recursive}.
The advantages afforded by the modulo~2 property are technical
and will become clear later in this paper.

Our main tool for building a non-P-recursive two-stack automaton
is the following result.

\begin{thm}[{\cite[Lemma~2]{GP1}}]\label{l:mod2}
Let~$\{a_n\}$ be a P-recursive integer sequence. Consider an
infinite binary word $\oa=(\al_1\al_2\ldots)$, defined by
$\al_n:= a_n\hskip.15cm{\rm mod}\hskip.15cm  2$.  Then,
there exists a finite binary word which is not a subword of~$\oa$.
\end{thm}

What follows is a construction of a two-stack automaton such that
the corresponding binary sequence~$\oa$ contains every finite
binary subword by design.  There are many such automata, in fact.
We give a complete description of this one as we need both its
notation and additional properties of the construction later on.

\subsection{The setup} \label{ss:2SA-setup}
Let~$\Gamma$ be a finite directed graph with vertices~$v_1,\ldots, v_m$.
Let~$X$ denote the set of labels of the form~$x_i$ and let~$X^{-1}$ denote
the set of labels of the form~$x_{i}^{-1}$, where~$i$ is any integer.
Define~$Y$ and~$Y^{-1}$ similarly. Label each vertex of~$\Gamma$ with
an element of~$X\cup X^{-1}\cup Y \cup Y^{-1}\cup\{\varepsilon\}$.

Let~$\ella(v)$ denote the label on vertex~$v$.  We say that~$w_1\sim w_2$
if~$w_1,w_2\in X\cup X^{-1}$ or if~$w_1,w_2\in Y\cup Y^{-1}$.
If~$\Gamma$ has an edge from~$v_i$ to~$v_j$, we say that~$v_i\rightarrow v_j$.

Contrary to standard notation, we refer to a path~$\gamma=\gamma_1\ldots\gamma_n$,
where each~$\gamma_i$ is a vertex, not an edge, and we say that such
a path is of length~$n$, even though it only has~$n-1$ edges.

We further require that $\ella(v_1)=\ella(v_2)=\varepsilon,$ and that
there is no edge $v_i\rightarrow v_j$, with ${\ella(v_i)\sim\ella(v_j)}$.
A graph~$\Gamma$ satisfying all of the above conditions is called a
{\em two-stack automaton}.

As we traverse a path~$\gamma$, we keep track of two words~$w_X\in X^\star $
and~$w_Y\in Y^\star $, which start out empty. Whenever we enter a vertex with
label~$x_i$, we append~$x_i$ to the end of~$w_X$. When we enter a vertex with
label~$x_i^{-1}$, we remove~$x_i$ from the end of~$w_X$.
We modify~$w_Y$ similarly when entering vertices with label~$y_i$ or~$y_i^{-1}$.
When we enter a vertex with label~$\varepsilon$, we do nothing.
A path is called {\em balanced} if every step of this process
is well defined and both~$w_X$ and~$w_Y$ are empty at the end of the path.
Let~$G(\Gamma, n)$ denote the number of balanced paths in~$\Gamma$
from~$v_1$ to~$v_2$ of length~$n$.

Define an involution~$\pi_\gamma\in S_n$ as follows. If the above process
writes an instance of a label to~$w_X$ or~$w_Y$ at some time~$t_i$,
and removes the same instance of that label for the first time at time~$t_j$,
then $\pi(t_i)=t_j$ and $\pi(t_j)=t_i$. If the process does not write
or remove anything at a time step~$t_k$, then $\pi(t_k)=t_k$.
For example, if
$\ella(\gamma_1)\ella(\gamma_2)\ldots\ella(\gamma_9)=
\varepsilon x_1y_1x_1y_1^{-1}x_1^{-1}\varepsilon x_1^{-1}\varepsilon$,
then~$\pi_\gamma=(2\ 8)(3\ 5)(4\ 6)$.
This gives the following alternate characterization of balanced paths.

\begin{prop}\label{prop}
A path~$\gamma$ is balanced if and only if there exists an
involution~$\pi_\gamma\in S_n$ such that:
\begin{enumerate}
\item~$\ella(\gamma_{i})=\varepsilon$ for all~$\pi_\gamma(i)=i$ ,
\item~$\ella(\gamma_{i})\in X\cup Y$ and~$\ella(\gamma_{\pi_\gamma(i)})=\ella(\gamma_{i})^{-1}$, for all~$\pi_\gamma(i)>i$, and
\item There are no~$i$ and~$j$ with~$\ella(\gamma_{i})\sim\ella(\gamma_j)$ such that~$i<j<\pi_\gamma(i)<\pi_\gamma(j)$.
\end{enumerate}
Further, this involution~$\pi_\gamma$ is uniquely defined for each balanced~$\gamma$.
\end{prop}

The proof is straightforward.

\subsection{Non-P-recursive automaton}\label{ss:2SA-bad}
We are now ready to present a construction of such automaton~$\Ga_1$,
which is given in Figure~\ref{gamma2}. The construction is based on
a smaller automaton $\Ga_2$ we introduced in~\cite{GP1}.

\hskip3.4cm
\begin{figure}
\begin{subfigure}[b]{.48\textwidth}
\tikzset{every state/.style={minimum size=28pt}}
\begin{tikzpicture}[shorten >=1pt,node distance=2cm,on grid,auto,scale=.55,every state/.append style={transform shape}]
   	\node[state,fill=lightgray,xshift=2\ml,yshift=-2\ml] (s1)   {\huge$\varepsilon_1$};
	\node[state,fill=lightgray,xshift=7\ml,yshift=-2\ml] (s2)    {\huge$\varepsilon_2$};
	\node[state,fill=lightgray,xshift=12\ml,yshift=-2\ml] (s3)   {\huge$\varepsilon_3$};
	\node[state,fill=lightgray,xshift=2\ml,yshift=-6\ml] (s6)    {\huge$\varepsilon_6$};
	\node[state,fill=lightgray,xshift=7\ml,yshift=-6\ml] (s5)    {\huge$\varepsilon_5$};
	\node[state,fill=lightgray,xshift=12\ml,yshift=-6\ml] (s4)    {\huge$\varepsilon_4$};
	\node[state,fill=lightgray,xshift=2\ml,yshift=-10\ml] (s8)    {\huge$\varepsilon_8$};
	\node[state,fill=lightgray,xshift=12\ml,yshift=-10\ml] (s7)   {\huge$\varepsilon_7$};
	\node[state,xshift=8.666\ml,yshift=-3.333\ml] (p2)   {$y_1$};
	\node[state,xshift=10.333\ml,yshift=-4.666\ml] (p3)   {$x_2^{-1}$};
	\node[state,xshift=8.666\ml,yshift=-2\ml] (p4)   {$x_0^{-1}$};
	\node[state,xshift=10.333\ml,yshift=-2\ml] (p5)   {$y_1$};
	\node[state,xshift=12\ml,yshift=-4\ml] (p6)   {$x_2^{-1}$};
	\node[state,xshift=9.5\ml,yshift=-6\ml] (p7)   {$y_2^{-1}$};
	\node[state,xshift=2\ml,yshift=-8\ml] (p8)   {$x_1^{-1}$};
	\node[state,xshift=1\ml,yshift=0\ml] (v1)   {$x_2$};
	\node[state,xshift=3\ml,yshift=0\ml] (v2)   {$y_2$};
	\node[state,xshift=6\ml,yshift=0\ml] (v3)   {$x_1^{-1}$};
	\node[state,xshift=8\ml,yshift=0\ml] (v4)   {$y_0$};
	\node[state,xshift=11\ml,yshift=0\ml] (v5)   {$x_1^{-1}$};
	\node[state,xshift=13\ml,yshift=0\ml] (v6)   {$y_1$};
	\node[state,xshift=14\ml,yshift=-1\ml] (v7)   {$x_0^{-1}$};
	\node[state,xshift=14\ml,yshift=-3\ml] (v8)   {$y_0$};
	\node[state,xshift=14\ml,yshift=-5\ml] (v9)   {$y_0^{-1}$};
	\node[state,xshift=14\ml,yshift=-7\ml] (v10)   {$x_0$};
	\node[state,xshift=11\ml,yshift=-8\ml] (v11)   {$y_1^{-1}$};
	\node[state,xshift=13\ml,yshift=-8\ml] (v12)   {$y_0$};
	\node[state,xshift=0\ml,yshift=-6\ml] (v13)   {$x_1^{-1}$};
	\node[state,xshift=2\ml,yshift=-4\ml] (v14)   {$x_0^{-1}$};
	\node[state,xshift=7\ml,yshift=-10\ml] (p0)   {$x_1^{-1}$};
	\node[state,xshift=7\ml,yshift=-12\ml] (p1)   {$x_0^{-1}$};
	\path[->]
    (s1) edge  node {} (s2)
    (s2) edge  node {} (p4)
    (p4) edge  node {} (p5)
    (p5) edge  node {} (s3)
    (s3) edge  node {} (p6)
    (p6) edge  node {} (s4)
    (s4) edge  node {} (p7)
    (p7) edge  node {} (s5)
    (s5) edge  node {} (s2)
    (s2) edge  node {} (p2)
    (p2) edge  node {} (p3)
    (p3) edge  node {} (s4)
    (s5) edge  node {} (s6)
    (s6) edge  node {} (p8)
    (p8) edge  node {} (s8)
    (s8) edge  node {} (p0)
    (p0) edge  node {} (s7)
    (s8) edge  node {} (p1)
    (p1) edge  node {} (s7)
    (s7) edge [bend right=30] node {} (s8)
    (s1) edge  node {} (v1)
    (v1) edge  node {} (v2)
    (v2) edge  node {} (s1)
    (s2) edge  node {} (v3)
    (v3) edge  node {} (v4)
    (v4) edge  node {} (s2)
    (s3) edge  node {} (v5)
    (v5) edge  node {} (v6)
    (v6) edge  node {} (s3)
    (s3) edge  node {} (v7)
    (v7) edge  node {} (v8)
    (v8) edge  node {} (s3)
    (s4) edge  node {} (v9)
    (v9) edge  node {} (v10)
    (v10) edge  node {} (s4)
    (s4) edge  node {} (v12)
    (v12) edge  node {} (v11)
    (v11) edge  node {} (s4)
    (s6) edge [bend left=30] node {} (v13)
    (s6) edge [bend left=30] node {} (v14)
    (v13) edge [bend left=30] node {} (s6)
    (v14) edge [bend left=30] node {} (s6)
       ;
\end{tikzpicture}
\end{subfigure}
\quad
\begin{subfigure}[b]{.48\textwidth}
\begin{tikzpicture}[shorten >=1pt,node distance=2.3cm,on grid,auto,scale=.80,every node/.append style={transform shape}]
      \large
      \node[state] (a)   {$s_1$};
   \node[state] (b) [right=of a] {$s_2$};
   \node[state] (c) [right=of b] {$s_3$};
   \node[state](d) [below =of c] {$s_4$};
   \node[state](e) [left=of d] {$s_5$};
   \node[state](f) [left=of e] {$s_6$};
   \node[state](g) [below=of d] {$s_7$};
   \node[state](h) [below=of f] {$s_8$};
    \path[->]
    (a) edge  node {} (b)
          edge  [loop above] node {\small~$xy$} ()
    (b) edge  node  {\small~$0_x^{-1}1_y$} (c)
    edge  node  {\small~$1_yx^{-1}$} (d)
          edge [loop above] node {\small~$1_x^{-1}0_y$} ()
    (c) edge  node {\small~$x^{-1}$} (d)
          edge [loop above] node {\small~$1_x^{-1}1_y$} ()
          edge [loop right] node {\small~$0_x^{-1}0_y$} ()
    (d) edge  node {\small~$y^{-1}$} (e)
          edge [loop below] node {\small~$1_y^{-1}1_x$} ()
          edge [loop right] node {\small~$0_y^{-1}0_x$} ()
    (e) edge node {} (b)
          edge node {} (f)
    (f) edge  node {\small~$1_x^{-1}$} (h)
          edge [loop above] node {\small~$1_x^{-1}$} ()
          edge [loop left] node {\small~$0_x^{-1}$} ()
    (g) edge [bend right=40] node {} (h)
    (h) edge node {\small~$1_x^{-1}$} (g)
         edge [bend right=40] node {\small~$0_x^{-1}$} (g);
\end{tikzpicture}
\end{subfigure}
\caption{The automata~$\Gamma_1$ (left), and~$\Gamma_2$ (right).}
\label{gamma2}
\end{figure}
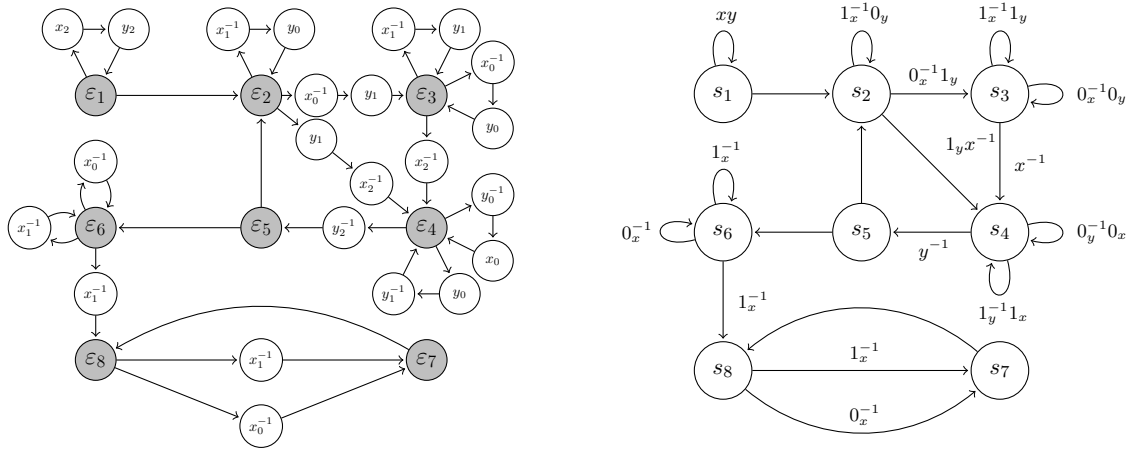

\begin{lem}\label{l:NZL}
There exists a two-stack automaton~$\Gamma_1$, such that
$\al_n:=G(\Gamma_1,n) \in \{0,1\}$ for all $n$, and such
that the word \ts $\oa = (\al_1\al_2\ldots)$ \ts is an infinite
binary word which contains every finite binary word as a subword.
\end{lem}

\begin{proof}
We give an explicit automaton~$\Gamma_1$ in Figure~1.
This automaton is formed by modifying the automaton~$\Gamma_2$,
given in~\cite{GP1}. Here we use $\varepsilon_1,\ldots,\varepsilon_8$
to denote the same trivial label~$\varepsilon$;
we make this distinction only for the purpose of illustration.
The vertex~$v_1$ is the shaded vertex labelled $\varepsilon_1$,
and the vertex~$v_2$ is the shaded vertex labelled~$\varepsilon_8$.
The vertex labelled~$\varepsilon_i$ in~$\Gamma_1$ corresponds to the
vertex labelled~$s_i$ in~$\Gamma_2$.

The primary difference between $\Gamma_1$ and $\Gamma_2$ is
that~$\Gamma_1$ has labels on vertices while~$\Gamma_2$ has labels on edges.
The labels were also changed by replacing~$0_x$,~$1_x$, and~$x$ with~$x_0$,~$x_1$,
and~$x_2$ respectively, and similarly for~$y$.  Since the lengths of the paths
change slightly, we get a slightly different formula, but the analysis is similar.

In counting paths we follow the proof of Lemma~3 in \cite{GP1}.
The valid paths through~$\Gamma_2$ have
$$
\mu \,= \, j+6k+2\sum_{i=1}^{k}\. \lfloor\log_2i\rfloor \quad \text{edges,}
$$
for some positive integers~$j$ and~$k$ such that the~$j$-th binary digit
of~$k$ is a~$1$. Every path through~$\Gamma_1$ will similarly
have $(\mu+1)$ vertices with label~$\varepsilon.$

Such paths will also have a total of~$4k$ vertices labelled~$x_2$,~$x_2^{-1}$,~$y_2$
or~$y_2^{-1}$, since~$k$ copies each of~$x_2$ and~$y_2$ are written and removed
in the computation. Similarly, every binary integer from 1 to~$k$ is written
and removed from both tapes, so the vertices with the remaining 8 labels are used
$$
\nu \, = \,
4k \. +\. 4\sum_{i=1}^{k}\. \lfloor\log_2i\rfloor \quad \text{times in total.}
$$

In summary, we have~$G(\Gamma_1,n)=1$ for all $n=(\mu+1) +\nu + 4k$,
where $j$ and~$k$ are positive integers such that the~$j$-th binary digit
of~$k$ is a~$1$, and~$G(\Gamma_1,n)=0$ otherwise. The word \ts
$\oa = (\al_1\al_2\ldots)$ \ts then contains the positive integer~$k$
written out in binary starting at
location
$$
n\, =\, 2\. + \. 14k \. + \. 6\. \sum_{i=1}^{k}\.\lfloor\log_2i\rfloor\.,
$$
and will therefore contain every finite binary word as a subword.

Finally, note that in the notion of ``valid path'' from~\cite{GP1},
it was possible for some instance of~$x^{-1}$ to cancel with a
later instance of~$x$.  For our purposes, this difference is
irrelevant since the words defined by paths in~$\Gamma_2$
do not have such cancellations.  \end{proof}

\bigskip

\section{Main Lemma and the proof of Theorem~\ref{t:NZT}}\label{s:proof-main-thm}

In this section we first change our setting from pattern avoidance to
slightly more general but equivalent notion of \emph{partial pattern avoidance}.
We state the Main Lemma~\ref{l:mainL} and show that it implies Theorem~\ref{t:NZT}.

\subsection{Partial patterns} \label{ss:proof-main-thm-partial-patterns}
A 0-1 matrix is called a {\em partial pattern} if
every row and column contains at most one~$1$. Clearly, every permutation pattern
is also a partial pattern. We say that a permutation matrix $M$ \emph{contains}
a partial pattern~$L$, if~$L$ can be obtained from~$M$ by deleting some
rows and columns; we say that $M$ {\em avoids}~$L$ otherwise. Given a
set~$\mathcal{F}$ of partial patterns, let $\cC_n(\cf)$ denote the set
of~$n\times n$ matrices~$M$ which avoid all partial patterns in~$\mathcal{F}$.
By analogy with the usual permutation patterns, let $C_{n}(\mathcal{F}) = |\cC_n(\cf)|$.

\begin{prop}\label{prop2}
Let $\cf_1$ be a finite set of partial patterns.  Then there exists
a finite set of the usual permutation patterns~$\cf_2$, such
that $C_{n}(\mathcal{F}_1)=C_n(\mathcal{F}_2)$ for all~$n \in \nn$.
\end{prop}

\begin{proof}
First, let us prove the result for a single partial pattern.
Let $L$ be a partial pattern of size~$p\times q$, and let $k=p+q$.
Denote by $\cp_{n}(L)$ be the set of $n\times n$ permutation matrices
containing~$L$.  Let us show by induction that
for all $n \ge k$, every permutation matrix $M\in \cp_n(L)$
contains a matrix in~$\cp_k(L)$.  Indeed, the claim is
trivially true for~$n=k$.  For larger~$n$, observe that every
$n\times n$ permutation matrix ~$M$ which contains~$L$ must
also contain some $i$-th~row and $j$-th~column, such that
$M_{p,q}=1$, and neither $i$-th~row nor $j$-th~column intersect~$L$.
This follows from the fact that otherwise \ts rank$(M) \le i+j <n$.
Deleting these row and column gives a smaller permutation matrix
which contains~$L$, proving the induction claim.

We conclude that \ts $\cf_2(L) := \cup_{\ell \le k} \cp_{\ell}(L)$ \ts
is the desired set of matrices for $\cf=\{L\}$.  In full generality,
take $\cf_2 = \cup_{L\in \cf} \cf_2(L)$.  The details are straightforward.
\end{proof}

It therefore suffices to disprove the Noonan--Zeilberger Conjecture~\ref{conj:NZT}
for partial patterns.

\begin{lem}[Main Lemma]\label{l:mainL}
Let~$\Gamma$ be a two-stack automaton. Then there exist
sets~$\mathcal{F}$ and~$\mathcal{F}^{\prime}$ of partial patterns,
and some integers $c,d\ge 1$, such that
$$
C_{cn+d}(\mathcal{F})-C_{cn+d}(\mathcal{F}^{\prime}) \. = \. G(\Gamma,n)
\hskip.15cm {\rm mod}\hskip.15cm  2\ts, \ \ \, \text{for all \ \. $n\in \nn$}\ts.
$$
\end{lem}

The proof of the Main Lemma is given in Section~\ref{s:MLP}.

\subsection{Proof of Theorem~\ref{t:NZT}}\label{ss:proof-main-thm-proof-NZT}
By Lemma~\ref{l:NZL}, there exists a two-stack automaton~$\Gamma_1$,
such that the infinite binary word $\oa=(\al_1\al_2\ldots)$ given by
$\al_n=G(\Gamma_1,n)$,
contains every finite binary word as a subword. By Lemma~\ref{l:mainL},
there exist integers~$c$ and~$d$ and two sets~$\mathcal{F}$
and~$\mathcal{F}^{\prime}$ of partial patterns such that
$C_{cn+d}(\mathcal{F})-C_{cn+d}(\mathcal{F}^{\prime})=
G(\Gamma_1,n) \hskip.15cm{\rm mod}\hskip.15cm  2$, for all~$n$.

If Conjecture~\ref{conj:NZT} is true, then both $\{C_n(\mathcal{F})\}$
and~$\{C_n(\mathcal{F}^{\prime})\}$ are P-recursive sequences.
Since P-recursive sequences are closed under taking the
differences and subsequences with indices in arithmetic progressions
(see e.g.~\cite[$\S$6.4]{Sta}),
this means that the sequence $\{a_n\}$, defined as \ts
$a_n=\{C_{cn+d}(\mathcal{F})-C_{cn+d}(\mathcal{F}^{\prime})\}$,
is also P-recursive.  On the other hand, from above, we have
$\al_n= a_n\hskip.15cm{\rm mod}\hskip.15cm 2$.
This gives a contradiction with Theorem~\ref{l:mod2}.

The second part of the theorem requires a quantitative form of the Main Lemma
and is given as Corollary~\ref{c:count}.  \qed

\bigskip

\section{The construction of an automaton in the Main Lemma}\label{s:contruction}

\subsection{Notation} \label{ss:construction-notation}
The construction of sets of matrices $\cf, \cf'$
has two layers and is quite involved, so
we try to simplify it by choosing a clear notation.
We use $\ca_g$ to denote a certain subset of $g\times g$ matrices,
which we call an \emph{alphabet} and use as building blocks.
We use English capitals with various decorations, notably
$A,A',B,B',E,L,P,Q,R,S,T_k$ and~$Z_p$, to
denote \ts $0$-$1$~matrices of size at most $g\times g$.
We use script capital letters \ts $\cf_i,\cf_i',\cw_i, \cw_i',$ \ts
to denote the sets of larger matrices (partial patterns) which form sets
$\cf,\cf'$. Each is of size at most $8\ts g\times 8\ts g$, and some of the
matrices are denoted $W_i$ and~$W_i'$.

On a bigger scale, we use $M=M(\ast,\ast)$
to denote large block matrices, with individual blocks $M^{i,j}$
being either zero or matrices in the alphabet~$\ca_g$.
For the proof of Theorem~\ref{t:NZT} we take $g=10$, but for
theorems~\ref{t:undecide} and~\ref{t:wilfian} we need larger~$g$.
When writing matrices, we use a dot $(\ts\cdot\ts{})$ within a matrix to represent
a single~0 entry, and a circle~$(\ts\circ\ts{})$ to represent a $g\times g$
submatrix of zeros.

\subsection{The alphabet}
A permutation matrix is called {\em simple} if it contains no permutation
matrix as a proper submatrix consisting of consecutive rows and columns,
other than the trivial $1\times 1$ permutation matrix.

Define an alphabet~$\mathcal{A}_g$ of all~$g\times g$ simple permutation
matrices which contain the following matrix as a submatrix:
{\small
$$
L\. = \. \left(
\arraycolsep=1pt\def\arraystretch{.8}\begin{array}{*{7}c}
\cdot & \cdot & \cdot & \cdot & 1 & \cdot & \cdot \\
\cdot & \cdot & \cdot & \cdot & \cdot & 1 & \cdot \\
\cdot & \cdot & \cdot & \cdot & \cdot & \cdot & 1 \\
\cdot & \cdot & \cdot & 1 & \cdot & \cdot & \cdot \\
1 & \cdot & \cdot & \cdot & \cdot & \cdot & \cdot \\
\cdot & 1 & \cdot & \cdot & \cdot & \cdot & \cdot \\
\cdot & \cdot & 1& \cdot & \cdot & \cdot & \cdot
\end{array}\right).
$$
}

\begin{prop}
We have: \. $|\ca_g|\.\to\. g!/e^{2}$ as $g \to \infty$.
\end{prop}

\begin{proof}
It was shown in~\cite{AAK} that the probability that a random $g\times g$
permutation matrix~$M$ is simple tends to $1/e^{2}$ as $g \to \infty$
(see also~\cite[A111111]{OEIS}). On the other hand, the probability
that $M$ avoids~$L$ tends to~$0$ as $g \to \infty$. Thus, the probability
that $M\in \mathcal{A}_g$ tends to~$1/e^2$, as desired. \end{proof}

By the proposition, we can fix an integer~$g$ large enough that
$|\mathcal{A}_g|> 5+m+r$, where $m$ is the number of vertices in
$\Gamma$ and $r$ is the number of distinct labels in $X\cup Y$
on vertices of~$\Gamma$.

\smallskip

We build our forbidden partial patterns out of elements of~$\mathcal{A}_g$
as follows.
Choose five special matrices $P,Q,B,B^\prime,E\in \mathcal{A}_g$,
as well as two classes of matrices,~$T_1,\ldots,T_{m}\in \mathcal{A}_g$,
and~${Z_1,\ldots, Z_r\in \mathcal{A}_g}$.  Here the matrices
$T_1,\ldots, T_m$ represent the~$m$ vertices in~$\Gamma$.
Let~$T_i$ denote the matrix corresponding to~$v_i$. The matrices
$Z_1,\dots,Z_r$ represent the~$r$ labels in $X\cup Y$.
Let~$s(Z_p)$ denote the label which corresponds to~$Z_p$.
Let us emphasize that these choices are arbitrary as the only important
properties of these $(5+m+r)$ matrices is that they are all in
$\mathcal{A}_g$ and distinct.

\subsection{Forbidden matrices}
Let~$\mathcal{F}_1$ denote the set of all $g\times (g+1)$ or $(g+1)\times g$
partial patterns formed by taking a matrix~$A$ in~$\mathcal{A}_g$,
and inserting a row or column of all zeros somewhere in the middle of~$A$.

Let~$\mathcal{F}_2$ denote the set of all $(2g+1)\times (5g+1)$ or $(5g+1)\times (2g+1)$
partial patterns whose bottom left $g\times g$ consecutive submatrix is
a $B$ or $B^\prime$ and whose top right $g\times g$ consecutive submatrix
is $T_j$ for some~$j$.

Let~$\mathcal{F}_3$ denote the four element set consisting of the $(2g+1)\times g$
and $g\times (2g+1)$ partial pattern formed by inserting $g+1$ rows of zeros below~$Q$,
inserting $g+1$ rows of zeros above $P$,  inserting $g+1$ columns of zeros to
the right of~$Q$, or inserting $g+1$ columns of zeros to the left of~$P$.

Let \ts $\mathcal{F}_4 = \mathcal{W}_1\cup\mathcal{W}_2\cup\mathcal{W}_3\cup\mathcal{W}_4\cup\mathcal{W}_5$\ts,
where $\cw_i$ are defined as follows.  Let $\mathcal{W}_1$, $\mathcal{W}_2$ and $\mathcal{W}_3$ denote the sets of
matrices of the form $W_1$, $W_2$ and $W_3$, respectively:
{\small
\vskip-.1cm
$$
{W_1 = \left(
\arraycolsep=1pt\def\arraystretch{.8}\begin{array}{*{8}c}
\circ&\circ&T_i&\circ&\circ&\circ&\circ&\circ\\
\circ&\circ&\circ&\circ&\circ&T_j&\circ&\circ\\
L&\circ&\circ&\circ&\circ&\circ&\circ&\circ\\
\circ&\circ&\circ&\circ&\circ&\circ&\circ&Z_p\\
\circ&\circ&\circ&\circ&\circ&\circ&T_k&\circ\\
\circ&B^\prime&\circ&\circ&\circ&\circ&\circ&\circ\\
\circ&\circ&\circ&\circ&R&\circ&\circ&\circ\\
\circ&\circ&\circ&Z_p&\circ&\circ&\circ&\circ
\end{array}\right), \ 
W_2 = \left(
\arraycolsep=1pt\def\arraystretch{.8}\begin{array}{*{8}c}
\circ&\circ&\circ&\circ&Z_p&\circ&\circ&\circ\\
\circ&\circ&\circ&T_i&\circ&\circ&\circ&\circ\\
\circ&\circ&\circ&\circ&\circ&\circ&T_j&\circ\\
\circ&L&\circ&\circ&\circ&\circ&\circ&\circ\\
Z_p&\circ&\circ&\circ&\circ&\circ&\circ&\circ\\
\circ&\circ&\circ&\circ&\circ&\circ&\circ&T_k\\
\circ&\circ&B^\prime&\circ&\circ&\circ&\circ&\circ\\
\circ&\circ&\circ&\circ&\circ&R&\circ&\circ
\end{array}\right), \ 
W_3 =\left(
\arraycolsep=1pt\def\arraystretch{.8}\begin{array}{*{7}c}
\circ&\circ&T_i&\circ&\circ&\circ&\circ\\
\circ&\circ&\circ&\circ&\circ&T_j&\circ\\
L&\circ&\circ&\circ&\circ&\circ&\circ\\
\circ&\circ&\circ&E&\circ&\circ&\circ\\
\circ&\circ&\circ&\circ&\circ&\circ&T_k\\
\circ&B^\prime&\circ&\circ&\circ&\circ&\circ\\
\circ&\circ&\circ&\circ&R&\circ&\circ
\end{array}\right)}.
$$
}

\nin
In all three cases, we require $L,R\in\{B,B^\prime\}$ and~$v_i\rightarrow v_j\rightarrow v_k$.
In $W_1$, we require $\ella(v_j)=s(Z_p)$. In $W_2$, we require $\ella(v_j)=(s(Z_p))^{-1}$.
In $W_3,$ we require $\ella(v_j)=\varepsilon$.

Similarly, let $\mathcal{W}_4$ and $\mathcal{W}_5$ denote the set of all matrices
of the form $W_4$ and $W_5$ respectively:
{\small
$$
{ W_4=\left(
\arraycolsep=2pt\def\arraystretch{.8}\begin{array}{*{6}c}
\circ&\circ&\circ&\circ&T_1&\circ\\
\circ&P&\circ&\circ&\circ&\circ\\
\circ&\circ&E&\circ&\circ&\circ\\
\circ&\circ&\circ&\circ&\circ&T_k\\
B^\prime&\circ&\circ&\circ&\circ&\circ\\
\circ&\circ&\circ&R&\circ&\circ
\end{array}\right)\hskip-.05cm,\hskip.2cm W_5=\left(
\arraycolsep=2pt\def\arraystretch{.8}\begin{array}{*{6}c}
\circ&\circ&T_i&\circ&\circ&\circ\\
\circ&\circ&\circ&\circ&\circ&T_2\\
L&\circ&\circ&\circ&\circ&\circ\\
\circ&\circ&\circ&E&\circ&\circ\\
\circ&\circ&\circ&\circ&Q&\circ\\
\circ&B^\prime&\circ&\circ&\circ&\circ
\end{array}\right)}.
$$
}

\nin
Here, in $W_4$, we require $R\in\{B,B^\prime\}$ and $v_1\rightarrow v_k$,
and in $W_5,$ we require $L\in\{B,B^\prime\}$ and $v_i\rightarrow v_2$.
Finally, let~$\mathcal{F}_5$ denote the set of all patterns of the form
{\small
$$
{\left(
\arraycolsep=2pt\def\arraystretch{.8}\begin{array}{*{3}c}
\circ&Z_p&\circ\\
\circ&\circ&Z_q\\
B&\circ&\circ
\end{array}\right)\.,\hskip.3cm
\left(\arraycolsep=2pt\def\arraystretch{.8}\begin{array}{*{3}c}
\circ&Z_p&\circ\\
\circ&\circ&Z_q\\
B^\prime&\circ&\circ
\end{array}\right)}\hskip.3cm \mbox{\ns or}\hskip.3cm
{\left(
\arraycolsep=2pt\def\arraystretch{.8}\begin{array}{*{3}c}
\circ&\circ&T_j\\
Z_p&\circ&\circ\\
\circ&Z_q&\circ
\end{array}\right)}\., \ \ \text{\ns where~$s(Z_p) \sim s(Z_q)$\ts.}
$$
}

\smallskip

\begin{lem}[Explicit Construction]\label{l:mainL2}
Given a two-stack automaton~$\Gamma$, let
$$
{\mathcal{F}:=\mathcal{F}_1\cup \mathcal{F}_2\cup \mathcal{F}_3 \cup \mathcal{F}_4\cup \mathcal{F}_5} \quad
\text{and} \quad \mathcal{F}^\prime:=\mathcal{F}\cup\{B,B^\prime\}\ts,
$$
where $\cf_1,\ldots,\cf_5,B,B'$ are defined as above. Then, for all $n$, we have:
$$
C_{\kK}(\mathcal{F})-C_{\kK}(\mathcal{F}^{\prime})=G(\Gamma,n) \hskip.15cm {\rm mod}\hskip.15cm  2, \ \ \text{where $\kK=(3n+2)g$\ts.}
$$
\end{lem}

The Main Lemma~\ref{l:mainL} follows immediately from this result.

\subsection{Counting Partial Patterns}
We will now analyze the above construction in the specific case of~$\Gamma_1$.

\begin{thm}\label{t:count}
There exists a set $\mathcal{F}$ of at most $6854$ partial patterns
of size at most $80\times 80$ such that $\{C_n(\mathcal{F})\}$
is not P-recursive.
\end{thm}

Converting these partial patterns into the usual permutation
patterns would require many more patterns.
However all the patterns avoided would still
be size at most $80\times 80$.

\begin{cor}\label{c:count}
There exists a set $\mathcal{F}$ of $80\times 80$ permutation
matrices such that the sequence $\{C_n(\mathcal{F})\}$ is not P-recursive.
In particular, $|\mathcal{F}| < 80! < 10^{119}$.
\end{cor}

\begin{proof}[Proof of Theorem~\ref{t:count}]
Observe that $\Gamma_1$ has $31$ vertices and uses $6$ labels in $X\cup Y$.
Therefore we need $5+31+6=42$ matrices in $\mathcal{A}_g$.
Let $g=10$. Consider the following simple $9\times 9$ pattern
{\small
$$L^\prime=\left(
\arraycolsep=1pt\def\arraystretch{.6}\begin{array}{*{9}c}
\cdot &\cdot & \cdot & \cdot & \cdot & 1 & \cdot & \cdot&\cdot  \\
\cdot  & 1&\cdot & \cdot & \cdot & \cdot & \cdot & \cdot&\cdot  \\
\cdot &\cdot & \cdot & \cdot & \cdot & \cdot & 1 & \cdot&\cdot  \\
\cdot &\cdot & \cdot & \cdot & \cdot & \cdot & \cdot &\cdot&1  \\
\cdot &\cdot & \cdot & \cdot & 1 & \cdot & \cdot & \cdot &\cdot \\
1 & \cdot & \cdot & \cdot & \cdot & \cdot & \cdot&\cdot&\cdot   \\
\cdot &\cdot  & 1 & \cdot & \cdot & \cdot & \cdot & \cdot&\cdot \\
\cdot &\cdot &\cdot & \cdot & \cdot & \cdot & \cdot & 1 &  \cdot \\
\cdot &\cdot & \cdot & 1& \cdot & \cdot & \cdot & \cdot&\cdot
\end{array}\right).
$$
}

Note that there are 60 ways to insert 1 into $L^\prime$ to form a simple
$10\times 10$ pattern. Indeed, the 1 may inserted anywhere other than
the 4 corners or the 36 locations that would form a $2\times 2$
consecutive submatrix. All 60 of these $10\times 10$ are distinct
and in $\mathcal{A}_g$.

For $\mathcal{F}_1$, we actually only need to include the 42 matrices
in $\mathcal{A}_{10}$ which are actually used, so
$|\mathcal{F}_1|=42\cdot 9\cdot 2=756$.  Similarly,
for $\mathcal{F}_2$, there are $2$ choices for the bottom left
$10\times 10$ consecutive submatrix and $31$ choices for the top
right $10\times 10$ consecutive submatrix. There are $41$ entries
in the middle and at most one of them can be a $1$, which can be
satisfied in $42$ ways. Therefore,
$|\mathcal{F}_2|=2\cdot 31\cdot 2\cdot 42=5208$.
Clearly, $|\mathcal{F}_3|=4$.

Let us show that $|\cf_4|=292$.  Indeed, a matrix in
$\mathcal{W}_1\cup\mathcal{W}_2\cup\mathcal{W}_3$
is defined by the path $v_i\rightarrow v_j\rightarrow v_k$ and
the choices for $L$ and $R$. There are $71$ paths in $\Gamma$ of length~$3$,
so we have $|\mathcal{W}_1\cup\mathcal{W}_2\cup\mathcal{W}_3|=71\cdot 4=284$.
A matrix in $\mathcal{W}_4$ is defined by the vertex $v_k$ and the choices for~$R$,
so we have $|\mathcal{W}_4|=4$.  Similarly, a matrix in $\mathcal{W}_5$ is defined by the
vertex $v_i$ and the choices for~$L$, so $|\mathcal{W}_5|=4$.

Finally, for $\mathcal{F}_5$, there are $6$ choices for $Z_p$, $3$ choices for $Z_q$ and
$31+2$ choices for the $B$, $B^\prime$ or~$T_j$. Therefore,
$|\mathcal{F}_5|=6\cdot 3\cdot 33=594$.
In total, $\mathcal{F}$ consists of \ts
$|\cf| = 756+5208+4+292+594=6854$ \ts partial patterns
of dimensions at most $80\times 80$. The set $\mathcal{F}^\prime$ has two extra
matrices, but can be made smaller than $\cf$ since avoiding $B^\prime$ makes all
matrices $W_i \in \mathcal{F}_4$ redundant.
\end{proof}

\bigskip

\section{Proof of the Explicit Construction Lemma~\ref{l:mainL2}}\label{s:MLP}

\nin
In this section we give a proof of Lemma~\ref{l:mainL2} by reducing it to
three technical lemmas which are proved in Section~\ref{s:TL}.
Briefly, since $\cf \ssu \cf'$, we have $\cC_n(\cf') \subseteq \cC_n(\cf)$
for all $n$.  Denote $\cd_n = \cC_n(\cf') \sm \cC_n(\cf)$.
We construct an explicit involution $\phi$ on~$\cd_n$
and analyze the set of fixed points~$\cd_n'$.  We show that the set
$\cd_n'$ has a very rigid structure emulating the working
of a given two-stack automaton~$\Gamma$.

\subsection{Preliminaries} The key idea of an involution~$\phi$
defined below is a switch $B \lra B'$ between submatrices $B$ and~$B'$,
in such a way that the fixed points~$\cd_n'$ of $\phi$ avoid~$B'$.
The remaining copies of $B$ create a general diagonal structure
of the matrices in~$\cd_n'$, and enforce the location of all other
submatrices from the alphabet.  We invite the reader to consult the
example in the next section to have a visual understanding of our
approach.

We also need a convenient notion of a \emph{marked submatrix}.
Such marked submatrix will always be a~$B$, and is located at a
specific position in forbidden matrices~$W_i$.  This is best illustrated
in the matrix formulas below, where marked submatrix~$B$ is boxed.

\smallskip

Let \ts $\mathcal{F}_4^\prime =
\mathcal{W}_1^\prime\cup\mathcal{W}_2^\prime\cup
\mathcal{W}_3^\prime\cup\mathcal{W}_4^\prime\cup
\mathcal{W}_5^\prime$, where $\cw_i' \ssv \cw_i$ are defined
to have no submatrices~$B'$ (so $L=R=B$ in the notation above),
and where the elements of $\cw_4'$ have a unique marked
submatrix~$B$.  Precisely, let
$\mathcal{W}^\prime_1$, $\mathcal{W}^\prime_2$
and $\mathcal{W}^\prime_3$ denote the set of matrices
of the form $W_1^\prime$, $W_2^\prime$ and $W_3^\prime$,
respectively:
{\small
$$
{ W^\prime_1=\left(\arraycolsep=1.7pt\def\arraystretch{.8}\begin{array}{*{8}c}
\circ&\circ&T_i&\circ&\circ&\circ&\circ&\circ\\
\circ&\circ&\circ&\circ&\circ&T_j&\circ&\circ\\
B&\circ&\circ&\circ&\circ&\circ&\circ&\circ\\
\circ&\circ&\circ&\circ&\circ&\circ&\circ&Z_p\\
\circ&\circ&\circ&\circ&\circ&\circ&T_k&\circ\\
\circ&\bbs &\circ&\circ&\circ&\circ&\circ&\circ\\
\circ&\circ&\circ&\circ&B&\circ&\circ&\circ\\
\circ&\circ&\circ&Z_p&\circ&\circ&\circ&\circ
\end{array}\right)\mts\hskip-.05cm,\hskip.05cm W^\prime_2=\left(
\arraycolsep=1.7pt\def\arraystretch{.8}\begin{array}{*{8}c}
\circ&\circ&\circ&\circ&Z_p&\circ&\circ&\circ\\
\circ&\circ&\circ&T_i&\circ&\circ&\circ&\circ\\
\circ&\circ&\circ&\circ&\circ&\circ&T_j&\circ\\
\circ&B&\circ&\circ&\circ&\circ&\circ&\circ\\
Z_p&\circ&\circ&\circ&\circ&\circ&\circ&\circ\\
\circ&\circ&\circ&\circ&\circ&\circ&\circ&T_k\\
\circ&\circ&\bbs &\circ&\circ&\circ&\circ&\circ\\
\circ&\circ&\circ&\circ&\circ&B&\circ&\circ
\end{array}\right)\mts\hskip-.05cm,\hskip.05cm  W^\prime_3=\left(
\arraycolsep=1.7pt\def\arraystretch{.8}\begin{array}{*{8}c}
\circ&\circ&T_i&\circ&\circ&\circ&\circ\\
\circ&\circ&\circ&\circ&\circ&T_j&\circ\\
B&\circ&\circ&\circ&\circ&\circ&\circ\\
\circ&\circ&\circ&E&\circ&\circ&\circ\\
\circ&\circ&\circ&\circ&\circ&\circ&T_k\\
\circ&\bbs &\circ&\circ&\circ&\circ&\circ\\
\circ&\circ&\circ&\circ&B&\circ&\circ
\end{array}\right)}.
$$
}
Similarly, let $\mathcal{W}_4^\prime$, and $\mathcal{W}_5^\prime$ denote
the set of matrices of the form $W_4^\prime$ and $W_5^\prime$, respectively:
{\small
$$
{ W^\prime_4=\left(
\arraycolsep=2pt\def\arraystretch{.8}\begin{array}{*{6}c}
\circ&\circ&\circ&\circ&T_1&\circ\\
\circ&P&\circ&\circ&\circ&\circ\\
\circ&\circ&E&\circ&\circ&\circ\\
\circ&\circ&\circ&\circ&\circ&T_k\\
\bbs &\circ&\circ&\circ&\circ&\circ\\
\circ&\circ&\circ&B&\circ&\circ
\end{array}\right)\hskip-.05cm,\hskip.2cm W^\prime_5=\left(
\arraycolsep=2pt\def\arraystretch{.8}\begin{array}{*{6}c}
\circ&\circ&T_i&\circ&\circ&\circ\\
\circ&\circ&\circ&\circ&\circ&T_2\\
B&\circ&\circ&\circ&\circ&\circ\\
\circ&\circ&\circ&E&\circ&\circ\\
\circ&\circ&\circ&\circ&Q&\circ\\
\circ&\bbs &\circ&\circ&\circ&\circ
\end{array}\right)}.
$$
}
Of course, all these $W_i'$ satisfy the same conditions
as $W_i$ in the previous section.

\subsection{Construction of the involution~$\phi$} \ts
From this point on, let $\kK=(3n+2)g$. Given a $\kK\times \kK$
permutation matrix~$M$, let $M^{i,j}$ refer to the $g\times g$~submatrix in
rows $g(i-1)+1$ through~$g\ts i$, and columns $g(j-1)+1$ through~$g\ts j$.

First, observe that
$\cd_n:= \cC_{\kK}(\mathcal{F})\sm \cC_{\kK}(\mathcal{F}^{\prime})$
is the set of all~$\kK\times \kK$ matrices avoiding~$\mathcal{F}$
with at least one submatrix~$B$ or~$B^\prime$.
Since~$\cd_n$ avoids~$\mathcal{F}_1$, every submatrix $B$ or~$B^\prime$
in a matrix in~$\cd_n$ must be a consecutive $g\times g$ block.
A submatrix $B$ or~$B^\prime$ in a matrix~$M$ of~$\cd_n$ is
called {\em blocked} if replacing it with $B^\prime$ or~$B$, respectively,
would result in a matrix not in~$\cd_n$.

\begin{lem}\label{TL1}
Consider the map~$\phi$ on~$\cd_n$ which takes the leftmost unblocked
submatrix~$B$ or~$B^\prime$ and replaces it with~$B^\prime$ or~$B$ respectively.
The map~$\phi$ is an involution on~$\cd_n$.  Furthermore, the fixed points
of~$\phi$ are the~$\kK\times \kK$ matrices~$M$ such that:
\begin{enumerate}
\item matrix~$M$ avoids~$\mathcal{F}$,
\item matrix~$M$ avoids~$B^\prime$,
\item matrix~$B$ is a submatrix of~$M$, and
\item every submatrix~$B$ inside~$M$ is a marked submatrix
of a matrix in~$\mathcal{F}_4^\prime$.
\end{enumerate}
\end{lem}

\noindent
Let~$\cd_n^\prime={\rm Fix}(\phi)$ denote the set of all fixed points of~$\phi$.
Since~$\phi$ is an involution, we conclude that
${|\cd_n|=|\cd_n^\prime| \hskip.15cm{\rm mod}\hskip.15cm 2}$,
so it suffices to show that
$|\cd_n^\prime|=G(\Gamma,n) \hskip.15cm{\rm mod}\hskip.15cm 2$.

\subsection{The structure of $\cd_n^\prime$}
Let~$\gamma$ be a path from~$v_1$ to~$v_2$ which is not necessarily balanced,
and let $\pi\in S_n$. Denote by $M:=M(\gamma,\pi)$ the $\kK\times \kK$
permutation matrix given by:
\begin{enumerate}
\item~$M^{2,2}=P$,
\item~$M^{3n+1,3n+1}=Q$,
\item~$M^{3i+2,3i-2}=B$ for all~$i$,
\item~$M^{3i-2,3i+2}=T_j$ for all~$i$, where~$\gamma_i=v_j$,
\item~$M^{3i,3j}=E$ whenever~$\ella(\gamma_i)=\varepsilon$ and~$\pi(i)=j$,
\item~$M^{3i,3j}=Z_p$ whenever~$\ella(\gamma_i)=s(Z_p)$ and~$\pi(i)=j$,
\item~$M^{3i,3j}=Z_p$ whenever~$\ella(\gamma_i)=(s(Z_p))^{-1}$ and~$\pi(i)=j$,
\item~$M^{i,j}=0$ is a zero matrix otherwise.
\end{enumerate}

\begin{lem}\label{TL2}
Every matrix in~$\cd_n^\prime$ is of the form $M(\gamma,\pi)$, where $\gamma$
is a path of length $n$ in $\Gamma$, and $\pi\in S_n$.
\end{lem}

Note, however, not every matrix~$M(\gamma,\pi)$ is in $\cd_n^\prime$.
The following lemma gives a complete characterization.
Recall that given a balanced path $\gamma$, there is a
unique permutation~$\pi_\gamma$ associated with $\gamma$ given in
Proposition~\ref{prop}.

\begin{lem}\label{TL3}
Let~$\gamma$ be a path in~$\Gamma$ of length $n$, and let~$\pi\in S_n$.
Then,~$M(\gamma,\pi)\in \cd_n^\prime$ if and only if~$\gamma$ is balanced
and~$\pi=\pi_\gamma$.
\end{lem}

Lemmas~\ref{TL1},~\ref{TL2} and~\ref{TL3} easily imply the Main Lemma.

\subsection{Proof of Lemma~\ref{l:mainL2}}
We have $C_{\kK}(\mathcal{F})-C_{\kK}(\mathcal{F}^{\prime})=|\cd_n|$ by definition.
Lemma~$\ref{TL1}$ shows that $\phi$ is an involution,
so $|\cd_n|=|{\rm Fix} (\phi)|=|\cd_n^\prime| \hskip.15cm{\rm mod}\hskip.15cm 2.$
Combining lemmas~\ref{TL2} and~\ref{TL3}, we get $|\cd_n^\prime|=G(\Gamma,n)$.
Thus, $C_{cn+d}(\mathcal{F})-C_{cn+d}(\mathcal{F}^{\prime})=G(\Gamma,n)
\hskip.15cm{\rm mod}\hskip.15cm 2$, as desired.\qed

\bigskip

\section{Example}\label{s:example}

\nin
Let us illustrate the construction in a simple case.
Consider a two-stack automaton~$\Gamma_3$ given in Figure~2.
Note that~$\Gamma_3$ has a unique balanced path
$\gamma=v_1v_3v_5v_3v_6v_4v_2v_4v_2$.

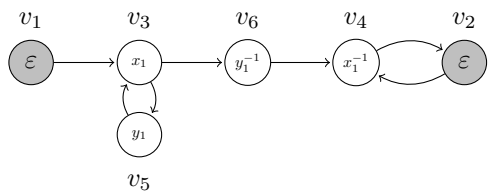
\begin{figure}[hbt]
\tikzset{every state/.style={minimum size=28pt}}
\begin{tikzpicture}[shorten >=1pt,node distance=2cm,on grid,auto,scale=.6,every state/.append style={transform shape}]
       \node[draw=white,xshift=0\ml,yshift=.7\ml] {$v_1$};
       \node[draw=white,xshift=1.8\ml,yshift=.7\ml] {$v_3$};
       \node[draw=white,xshift=3.6\ml,yshift=.7\ml] {$v_6$};
       \node[draw=white,xshift=5.4\ml,yshift=.7\ml] {$v_4$};
       \node[draw=white,xshift=7.2\ml,yshift=.7\ml] {$v_2$};
       \node[draw=white,xshift=1.8\ml,yshift=-2\ml] {$v_5$};
   	\node[state,fill=lightgray,xshift=0\ml,yshift=0\ml] (s1)   {\huge$\varepsilon$};
	\node[state,xshift=3\ml,yshift=0\ml] (s2)   {$x_1$};
	\node[state,xshift=6\ml,yshift=0\ml] (s3)   {$y_1^{-1}$};
	\node[state,xshift=9\ml,yshift=0\ml] (s4)   {$x_1^{-1}$};
	\node[state,fill=lightgray,xshift=12\ml,yshift=0\ml] (s5)   {\huge$\varepsilon$};
	\node[state,xshift=3\ml,yshift=-2\ml] (s6)   {$y_1$};
	\path[->]
    (s1) edge  node {} (s2)
    (s2) edge  node {} (s3)
    (s3) edge  node {} (s4)
    (s4) edge  [bend left=30]  node {} (s5)
    (s5) edge  [bend left=30]  node {} (s4)
    (s6) edge  [bend left=30]  node {} (s2)
    (s2) edge  [bend left=30]  node {} (s6)
       ;
\end{tikzpicture}
\caption{The two-stack automaton~$\Gamma_3$.}
\end{figure}

Let us show that the following matrix $M=M(\gamma,\pi_\gamma)$
is unique in the set of fixed points~$\cd_9^\prime$.
Here in the matrix, we have~$s(Z_1)=x_1$ and~$s(Z_2)=y_1$ (see below).


As in the definition of $M(\gamma,\pi_\gamma)$, observe that $M$ has
a diagonal of $B$ entries below the main diagonal, and a diagonal of
$T_i$ entries above the main diagonal. The $T_i$ entries give
the vertices of the path $\gamma$, in order. Observe that $M$
also has a~$P$ in the top left and $Q$ in the bottom right,
as in the definition of $M(\gamma,\pi_\gamma)$.

We have here the involution $\pi_\gamma=(2\ 8)(3\ 5)(4\ 6)$, and the
locations of the $E$, $Z_1$ and $Z_2$ matrices form the permutation matrix
for $\pi_\gamma$. Each matrix $E$ corresponds to a time when the
path $\gamma$ visited a vertex labelled $\varepsilon$.
Similarly, each $Z_p$ above the diagonal corresponds
to a pair of times when a given instance of a label
was written and removed from one of the stacks.

{\small
$$M \, = \,
{\left(
\arraycolsep=1pt\def\arraystretch{.8}\begin{array}{*{29}c}
\circ & \circ & \circ & \circ & T_1 & \circ & \circ & \circ & \circ & \circ & \circ & \circ & \circ  & \circ & \circ & \circ & \circ & \circ & \circ & \circ & \circ & \circ & \circ & \circ & \circ & \circ & \circ & \circ & \circ\\
\circ & P & \circ & \circ & \circ & \circ & \circ & \circ & \circ & \circ & \circ & \circ & \circ  & \circ & \circ & \circ & \circ & \circ & \circ & \circ & \circ & \circ & \circ & \circ & \circ & \circ & \circ & \circ & \circ\\
\circ & \circ & E & \circ & \circ & \circ & \circ & \circ & \circ & \circ & \circ & \circ & \circ  & \circ & \circ & \circ & \circ& \circ & \circ & \circ & \circ & \circ & \circ & \circ & \circ & \circ & \circ & \circ & \circ\\
\circ & \circ & \circ & \circ & \circ & \circ & \circ & T_3& \circ & \circ & \circ & \circ & \circ  & \circ & \circ & \circ & \circ& \circ & \circ & \circ & \circ & \circ & \circ & \circ & \circ & \circ & \circ & \circ & \circ\\
B & \circ & \circ & \circ & \circ & \circ & \circ & \circ & \circ & \circ & \circ & \circ & \circ & \circ & \circ & \circ & \circ& \circ & \circ & \circ & \circ & \circ & \circ & \circ & \circ & \circ & \circ & \circ & \circ\\
\circ & \circ & \circ & \circ & \circ & \tikz\draw[red,fill=red] (0,0) circle (.5ex); & \tikz\draw[red,fill=red] (0,0) circle (.5ex); & \tikz\draw[red,fill=red] (0,0) circle (.5ex); & \tikz\draw[red,fill=red] (0,0) circle (.5ex); & \tikz\draw[red,fill=red] (0,0) circle (.5ex); & \tikz\draw[red,fill=red] (0,0) circle (.5ex); & \tikz\draw[red,fill=red] (0,0) circle (.5ex); & \tikz\draw[red,fill=red] (0,0) circle (.5ex); & \tikz\draw[red,fill=red] (0,0) circle (.5ex); & \tikz\draw[red,fill=red] (0,0) circle (.5ex); & \tikz\draw[red,fill=red] (0,0) circle (.5ex); & \tikz\draw[red,fill=red] (0,0) circle (.5ex);& \tikz\draw[red,fill=red] (0,0) circle (.5ex); & \tikz\draw[red,fill=red] (0,0) circle (.5ex); & \tikz\draw[red,fill=red] (0,0) circle (.5ex); & \tikz\draw[red,fill=red] (0,0) circle (.5ex); & \tikz\draw[red,fill=red] (0,0) circle (.5ex); & \tikz\draw[red,fill=red] (0,0) circle (.5ex); & {\color{red}Z_1}  & \circ & \circ & \circ & \circ & \circ\\
\circ & \circ & \circ & \circ & \circ & \tikz\draw[red,fill=red] (0,0) circle (.5ex); & \circ & \circ & \circ & \circ & T_5 & \circ & \circ & \circ & \circ & \circ & \circ & \circ & \circ & \circ & \circ & \circ & \circ & \tikz\draw[red,fill=red] (0,0) circle (.5ex); & \circ & \circ & \circ & \circ & \circ\\
\circ & \circ & \circ & B & \circ & \tikz\draw[red,fill=red] (0,0) circle (.5ex); & \circ & \circ & \circ & \circ & \circ & \circ & \circ & \circ & \circ & \circ & \circ& \circ & \circ & \circ & \circ & \circ & \circ & \tikz\draw[red,fill=red] (0,0) circle (.5ex); & \circ & \circ & \circ & \circ & \circ\\
\circ & \circ & \circ & \circ & \circ & \tikz\draw[red,fill=red] (0,0) circle (.5ex); & \circ & \circ & \tikz\draw[cyan,fill=cyan] (0,0) circle (.5ex); & \tikz\draw[cyan,fill=cyan] (0,0) circle (.5ex); & \tikz\draw[cyan,fill=cyan] (0,0) circle (.5ex); & \tikz\draw[cyan,fill=cyan] (0,0) circle (.5ex); & \tikz\draw[cyan,fill=cyan] (0,0) circle (.5ex); & \tikz\draw[cyan,fill=cyan] (0,0) circle (.5ex); & {\color{cyan}Z_2} & \circ & \circ & \circ & \circ & \circ & \circ & \circ & \circ & \tikz\draw[red,fill=red] (0,0) circle (.5ex); & \circ & \circ & \circ & \circ & \circ\\
\circ & \circ & \circ & \circ & \circ & \tikz\draw[red,fill=red] (0,0) circle (.5ex); & \circ & \circ & \tikz\draw[cyan,fill=cyan] (0,0) circle (.5ex); & \circ & \circ & \circ & \circ & T_3 & \tikz\draw[cyan,fill=cyan] (0,0) circle (.5ex); & \circ & \circ & \circ  & \circ & \circ & \circ & \circ & \circ & \tikz\draw[red,fill=red] (0,0) circle (.5ex); & \circ & \circ & \circ & \circ & \circ\\
\circ & \circ & \circ & \circ & \circ & \tikz\draw[red,fill=red] (0,0) circle (.5ex); & B & \circ & \tikz\draw[cyan,fill=cyan] (0,0) circle (.5ex); & \circ & \circ & \circ & \circ & \circ & \tikz\draw[cyan,fill=cyan] (0,0) circle (.5ex); & \circ & \circ & \circ & \circ & \circ & \circ & \circ & \circ & \tikz\draw[red,fill=red] (0,0) circle (.5ex);& \circ & \circ & \circ & \circ & \circ\\
\circ & \circ & \circ & \circ & \circ & \tikz\draw[red,fill=red] (0,0) circle (.5ex); & \circ & \circ & \tikz\draw[cyan,fill=cyan] (0,0) circle (.5ex); & \circ & \circ & \tikz\draw[red,fill=red] (0,0) circle (.5ex); & \tikz\draw[red,fill=red] (0,0) circle (.5ex); & \tikz\draw[red,fill=red] (0,0) circle (.5ex); & \tikz\draw[black,fill=black] (0,0) circle (.5ex); & \tikz\draw[red,fill=red] (0,0) circle (.5ex); & \tikz\draw[red,fill=red] (0,0) circle (.5ex); & {\color{red}Z_1}& \circ & \circ & \circ & \circ & \circ & \tikz\draw[red,fill=red] (0,0) circle (.5ex); & \circ & \circ & \circ & \circ & \circ\\
\circ & \circ & \circ & \circ & \circ & \tikz\draw[red,fill=red] (0,0) circle (.5ex); & \circ & \circ & \tikz\draw[cyan,fill=cyan] (0,0) circle (.5ex); & \circ & \circ & \tikz\draw[red,fill=red] (0,0) circle (.5ex); & \circ & \circ & \tikz\draw[cyan,fill=cyan] (0,0) circle (.5ex); & \circ & T_6 & \tikz\draw[red,fill=red] (0,0) circle (.5ex); & \circ & \circ& \circ & \circ & \circ & \tikz\draw[red,fill=red] (0,0) circle (.5ex); & \circ & \circ & \circ & \circ & \circ \\
\circ & \circ & \circ & \circ & \circ & \tikz\draw[red,fill=red] (0,0) circle (.5ex); & \circ & \circ & \tikz\draw[cyan,fill=cyan] (0,0) circle (.5ex); & B & \circ & \tikz\draw[red,fill=red] (0,0) circle (.5ex); & \circ & \circ & \tikz\draw[cyan,fill=cyan] (0,0) circle (.5ex); & \circ & \circ & \tikz\draw[red,fill=red] (0,0) circle (.5ex); & \circ & \circ& \circ & \circ & \circ & \tikz\draw[red,fill=red] (0,0) circle (.5ex); & \circ & \circ & \circ & \circ & \circ \\
\circ & \circ & \circ & \circ & \circ & \tikz\draw[red,fill=red] (0,0) circle (.5ex); & \circ & \circ & {\color{cyan}Z_2} & \tikz\draw[cyan,fill=cyan] (0,0) circle (.5ex); & \tikz\draw[cyan,fill=cyan] (0,0) circle (.5ex); & \tikz\draw[black,fill=black] (0,0) circle (.5ex); & \tikz\draw[cyan,fill=cyan] (0,0) circle (.5ex); & \tikz\draw[cyan,fill=cyan] (0,0) circle (.5ex); & \tikz\draw[cyan,fill=cyan] (0,0) circle (.5ex); & \circ & \circ & \tikz\draw[red,fill=red] (0,0) circle (.5ex); & \circ & \circ& \circ & \circ & \circ & \tikz\draw[red,fill=red] (0,0) circle (.5ex); & \circ & \circ & \circ & \circ & \circ \\
\circ & \circ & \circ & \circ & \circ & \tikz\draw[red,fill=red] (0,0) circle (.5ex); & \circ & \circ & \circ & \circ & \circ & \tikz\draw[red,fill=red] (0,0) circle (.5ex); & \circ & \circ & \circ & \circ & \circ & \tikz\draw[red,fill=red] (0,0) circle (.5ex); & \circ & T_4& \circ & \circ & \circ & \tikz\draw[red,fill=red] (0,0) circle (.5ex); & \circ & \circ & \circ & \circ & \circ \\
\circ & \circ & \circ & \circ & \circ & \tikz\draw[red,fill=red] (0,0) circle (.5ex); & \circ & \circ & \circ & \circ & \circ & \tikz\draw[red,fill=red] (0,0) circle (.5ex);& B & \circ & \circ & \circ & \circ & \tikz\draw[red,fill=red] (0,0) circle (.5ex); & \circ & \circ& \circ & \circ & \circ & \tikz\draw[red,fill=red] (0,0) circle (.5ex); & \circ & \circ & \circ & \circ & \circ \\
\circ & \circ & \circ & \circ & \circ & \tikz\draw[red,fill=red] (0,0) circle (.5ex); & \circ & \circ & \circ & \circ & \circ & {\color{red}Z_1} & \tikz\draw[red,fill=red] (0,0) circle (.5ex); & \tikz\draw[red,fill=red] (0,0) circle (.5ex); & \tikz\draw[red,fill=red] (0,0) circle (.5ex); & \tikz\draw[red,fill=red] (0,0) circle (.5ex); & \tikz\draw[red,fill=red] (0,0) circle (.5ex); & \tikz\draw[red,fill=red] (0,0) circle (.5ex); & \circ & \circ& \circ & \circ & \circ & \tikz\draw[red,fill=red] (0,0) circle (.5ex); & \circ & \circ & \circ & \circ & \circ \\
\circ & \circ & \circ & \circ & \circ & \tikz\draw[red,fill=red] (0,0) circle (.5ex); & \circ & \circ & \circ & \circ & \circ & \circ & \circ & \circ & \circ & \circ & \circ & \circ & \circ & \circ & \circ & \circ & T_2 & \tikz\draw[red,fill=red] (0,0) circle (.5ex); & \circ & \circ & \circ & \circ & \circ\\
\circ & \circ & \circ & \circ & \circ & \tikz\draw[red,fill=red] (0,0) circle (.5ex); & \circ & \circ & \circ & \circ & \circ & \circ & \circ & \circ & \circ & B & \circ & \circ & \circ & \circ & \circ & \circ & \circ & \tikz\draw[red,fill=red] (0,0) circle (.5ex); & \circ & \circ & \circ & \circ & \circ\\
\circ & \circ & \circ & \circ & \circ & \tikz\draw[red,fill=red] (0,0) circle (.5ex); & \circ & \circ & \circ & \circ & \circ & \circ & \circ & \circ & \circ & \circ & \circ & \circ & \circ & \circ & E & \circ & \circ & \tikz\draw[red,fill=red] (0,0) circle (.5ex);& \circ & \circ & \circ & \circ & \circ\\
\circ & \circ & \circ & \circ & \circ & \tikz\draw[red,fill=red] (0,0) circle (.5ex); & \circ & \circ & \circ & \circ & \circ & \circ & \circ & \circ & \circ & \circ & \circ & \circ & \circ & \circ & \circ & \circ & \circ & \tikz\draw[red,fill=red] (0,0) circle (.5ex); & \circ & T_4 & \circ & \circ & \circ\\
\circ & \circ & \circ & \circ & \circ & \tikz\draw[red,fill=red] (0,0) circle (.5ex); & \circ & \circ & \circ & \circ & \circ & \circ & \circ & \circ & \circ & \circ & \circ & \circ & B & \circ & \circ & \circ & \circ & \tikz\draw[red,fill=red] (0,0) circle (.5ex);& \circ & \circ & \circ & \circ & \circ\\
\circ & \circ & \circ & \circ & \circ & {\color{red}Z_1} & \tikz\draw[red,fill=red] (0,0) circle (.5ex); & \tikz\draw[red,fill=red] (0,0) circle (.5ex); & \tikz\draw[red,fill=red] (0,0) circle (.5ex); & \tikz\draw[red,fill=red] (0,0) circle (.5ex); & \tikz\draw[red,fill=red] (0,0) circle (.5ex); & \tikz\draw[red,fill=red] (0,0) circle (.5ex); & \tikz\draw[red,fill=red] (0,0) circle (.5ex); & \tikz\draw[red,fill=red] (0,0) circle (.5ex); & \tikz\draw[red,fill=red] (0,0) circle (.5ex); & \tikz\draw[red,fill=red] (0,0) circle (.5ex); & \tikz\draw[red,fill=red] (0,0) circle (.5ex); & \tikz\draw[red,fill=red] (0,0) circle (.5ex); & \tikz\draw[red,fill=red] (0,0) circle (.5ex); & \tikz\draw[red,fill=red] (0,0) circle (.5ex); & \tikz\draw[red,fill=red] (0,0) circle (.5ex); & \tikz\draw[red,fill=red] (0,0) circle (.5ex); & \tikz\draw[red,fill=red] (0,0) circle (.5ex); & \tikz\draw[red,fill=red] (0,0) circle (.5ex); & \circ & \circ & \circ & \circ & \circ\\
 \circ & \circ & \circ & \circ & \circ & \circ & \circ & \circ & \circ & \circ & \circ & \circ & \circ & \circ & \circ & \circ & \circ & \circ & \circ & \circ & \circ & \circ & \circ & \circ & \circ & \circ & \circ & \circ & T_2\\
\circ & \circ & \circ & \circ & \circ & \circ & \circ & \circ & \circ & \circ & \circ & \circ & \circ & \circ & \circ & \circ & \circ & \circ & \circ & \circ & \circ & B & \circ & \circ & \circ & \circ & \circ & \circ & \circ\\
\circ & \circ & \circ & \circ & \circ & \circ & \circ & \circ & \circ & \circ & \circ & \circ & \circ & \circ & \circ & \circ & \circ & \circ & \circ & \circ & \circ & \circ & \circ & \circ & \circ & \circ & E & \circ & \circ\\
\circ & \circ & \circ & \circ & \circ & \circ & \circ & \circ & \circ & \circ & \circ & \circ & \circ & \circ & \circ & \circ & \circ & \circ & \circ & \circ & \circ & \circ & \circ & \circ & \circ & \circ & \circ & Q & \circ\\
\circ & \circ & \circ & \circ & \circ & \circ & \circ & \circ & \circ & \circ & \circ & \circ & \circ & \circ & \circ & \circ & \circ & \circ & \circ & \circ & \circ & \circ & \circ & \circ & B& \circ & \circ & \circ & \circ
\end{array}\right)}.
$$
}

\smallskip

The red and blue squares in $M$ connect each $Z_p$ with
the corresponding times along $\gamma$ that the label was
written and removed. Red represents~$X$, while blue represents~$Y$,
as defined in Section~\ref{s:2SA}. Notice that when two of these
squares cross (marked black), it means that the first label written was also
the first label removed.  This can only happen when the two labels
are written on different stacks, so squares of the same color cannot cross.

The matrix~$M$ avoids~$\mathcal{F}_2$, since the only copes of~$P$
and~$Q$ are near the top left and bottom right corner. Similarly,
matrix~$M$ avoids~$\mathcal{F}_3$, since no~$T_i$ is too far up
and to the right of any~$B$. Clearly,~$M$ avoids~$B^\prime$,
so~$M$ avoids~$\mathcal{F}_4$.

Now recall the matrix~$L$ in the
definition of the alphabet~$\ca_{g}$:
{\small
$$L \, = \,
{\left(
\arraycolsep=2pt\def\arraystretch{.8}\begin{array}{*{7}c}
\cdot&\cdot&\cdot&\cdot&1&\cdot&\cdot\\
\cdot&\cdot&\cdot&\cdot&\cdot&1&\cdot\\
\cdot&\cdot&\cdot&\cdot&\cdot&\cdot&1\\
\cdot&\cdot&\cdot&1&\cdot&\cdot&\cdot\\
1&\cdot&\cdot&\cdot&\cdot&\cdot&\cdot\\
\cdot&1&\cdot&\cdot&\cdot&\cdot&\cdot\\
\cdot&\cdot&1&\cdot&\cdot&\cdot&\cdot
\end{array}\right)}.
$$}

\noindent
Observe that~$M$ avoids~$\mathcal{F}_1$, since there is no
submatrix where each~$1$ comes from a different~$g\times g$ block.
Finally, the fact that~$M$ avoids~$\mathcal{F}_5$ corresponds
to the fact that the lines coming out of $Z_p$ and $Z_q$ never
cross when~$Z_p\sim Z_q$.

Clearly, matrix~$M$ contains~$B$ and avoids~$B^\prime$. One can verify
that every submatrix~$B$ in~$M$ is a marked submatrix in some
matrix $\mathcal{F}^\prime_4$.  Since~$M$ satisfies all of the conditions
of Lemma~\ref{TL1}, we conclude that $M\in \cd_n^\prime$.

\bigskip

\section{Proofs of technical lemmas}\label{s:TL}

\subsection{Proof of Lemma~\ref{TL1}.}
First, let us show that every matrix~$M$ which satisfies the
following properties, is a fixed point of~$\phi$~:
\begin{enumerate}
\item matrix~$M$ avoids~$\mathcal{F}$,
\item matrix~$M$ avoids~$B^\prime$,
\item matrix~$B$ is a submatrix of~$M$, and
\item every submatrix~$B$ in~$M$ is a marked submatrix
in some matrix $\mathcal{F}_4^\prime$.
\end{enumerate}

\smallskip

\noindent
Observe that (1) and (3) imply that~$M$ avoids~$\mathcal{F}$ but not
$\cf'=\mathcal{F}\cup\{B,B^\prime\}$. Therefore, we have $M \in \cd_n$.
Further, (2)~and (4) imply that~$M$ is a fixed point of~$\phi$,
since $M$ has no~$B^\prime$, and replacing any~$B$ with a~$B^\prime$
will create a matrix in~$\mathcal{F}_4$.

Next, we show that every matrix in $\cd_n$ which violates the above criteria
is fixed by $\phi^2$ but not by~$\phi$. Observe that a matrix~$M\in \cd_n$
which satisfies~(2) also satisfies (1) and~(3). Consider now a
matrix~$M\in \cd_n$ which violates either (2) or (4).
Clearly, we have~$\phi(M)\in \cd_n$. It suffices to show that~$\phi(M)\neq M$
and that~$\phi^2(M)=M$.

Let~$N$ be a matrix in~$\cd_n$. Denote by~$A$ a submatrix~$B$ or~$B^\prime$
in~$N$. Let~$N^\prime$ be the matrix formed by replacing~$A$
with~$B^\prime$ or~$B$, respectively. Similarly, let~$A^\prime$ denote this
submatrix $B$ or~$B^\prime$ in~$N^\prime$.

Assume that~$N^\prime$ did not avoid~$\mathcal{F}_1$.
Then there would exist a submatrix $S\in \mathcal{A}_g$ of~$N^\prime$
which is not a consecutive~$g\times g$ block. Note that $S$ must intersect~$A^\prime$.
Consider the submatrix $S\cap A^\prime$. This must be a permutation matrix,
since~$S$ and~$A^\prime$ are permutation matrices. Furthermore, $S\cap A^\prime$
must be a consecutive submatrix of~$S$, since~$A^\prime$ is a consecutive
submatrix of~$N^\prime$. Since $S$ is simple, we have that~$S\cap A^\prime$
is a single~1 entry. This entry might be a 0 in~$N$, but it does not matter
because we could replace it with another 1 entry from~$A$, to form a
non-consecutive submatrix~$S$ in~$N$. Thus,~$N$ does not avoid~$\mathcal{F}_1$,
contradicting the assumption that~$N\in \cd_n$.

Since~$N^\prime$ avoids~$\mathcal{F}_1$, any instance of a matrix
from~$\mathcal{F}_2,$~$\mathcal{F}_3$, or~$\mathcal{F}_5$
intersecting~$A^\prime$ in~$N^\prime$ must contain~$A^\prime$ or
intersect~$A^\prime$ in a single row or column. In either case,
replacing~$A$ with~$A^\prime$ gives another matrix in~$\mathcal{F}_2,$
$\mathcal{F}_3$, or~$\mathcal{F}_5$, contradicting the fact that~$N^\prime$
avoids $\mathcal{F}_2$,~$\mathcal{F}_3$ and~$\mathcal{F}_5$.

Therefore, if the submatrix~$A$ is blocked it must be because~$A^\prime$
is contained in a matrix in~$\mathcal{F}_4$. Since~$\mathcal{F}_4$
is closed under replacing any~$B$ with a~$B^\prime$, it must be that~$A=B$
and~$A^\prime=B^\prime$. Any matrix which violates (2) therefore contains
a instance of~$B^\prime$, which is necessarily unblocked, and so is not
fixed by~$\phi$. For matrices in~$\cd_n$ which satisfy (2), containing
an unblocked instance of~$B$ is equivalent to violating (4). Therefore,
any matrix which violates (4) is not fixed by~$\phi$, so~$\phi(M)\neq M$.

Let~$\ord(N)$ denote the leftmost unblocked instance of~$B$ or~$B^\prime$ in~$N$,
and let~$A=\ord(N)$.  Observe that~$A^\prime$ is also clearly unblocked in~$N^\prime$.
Assume that~$A^\prime\neq \ord(N^\prime)$. There must be another unblocked
instance~$S$ of~$B$ or~$B^\prime$, which is further to the left than~$A^\prime$.
However,~$S$ would also be unblocked in~$N$, contradicting the assumption
that~$A=\ord(N)$. Therefore, $A^\prime=\ord(N^\prime)$, which implies that
$\phi(N^\prime)=N$.  We conclude that
$\phi(\phi(M))=M$ for all~$M$.

In summary, every matrix~$M$ satisfying (1) through~(4) is fixed by~$\phi$,
and every matrix in~$\cd_n$ not satisfying (1)~through (4) is fixed
by~$\phi^2$ but not $\phi$.  Therefore, map~$\phi$ is an involution
whose fixed points are exactly the matrices satisfying (1)
through~(4).\qed

\bigskip

\subsection{Proof of Lemma~\ref{TL2}.}
Let~$M$ be a matrix in~$\cd_n^\prime$. Lemma~\ref{TL1} shows that
\begin{enumerate}
\item matrix~$M$ avoids~$\mathcal{F}$,
\item matrix~$M$ avoids~$B^\prime$,
\item matrix~$B$ is a submatrix of~$M$, and
\item every instance of~$B$ in~$M$ is the marked submatrix
in some matrix in~$\mathcal{F}_4^\prime$.
\end{enumerate}

\noindent
Observe that every matrix in~$\mathcal{F}_4^\prime$ has a~$B$ or a~$P$ above
the marked submatrix~$\bbs$. Therefore, for every instance of~$B$ in~$M$,
there must be another~$B$ or a~$P$ somewhere above it. Since there
is at least one~$B$ in~$M$, there must be at least one~$P$ in~$M$.
This~$P$ must have at least $g$ rows above it and $g$ columns to the left,
since it is contained in some matrix in~$\mathcal{W}^\prime_4$.
Note that matrix~$P$ cannot have any more rows above it or columns to the left,
since~$M$ avoids~$\mathcal{F}_3$. Therefore,~$M^{2,2}$ is the unique
instance of~$P$ in~$M$. Similar analysis shows that~$M^{3n+1,3n+1}$
is the unique instance of~$Q$ in~$M$.

Similarly, every matrix in~$\mathcal{F}_4^\prime$ has a~$T_j$ at least~$4g$
rows above and $4g$ columns to the right of the~$\bbs$.
Since~$M$ avoids~$\mathcal{F}_2$, there can be no more rows or columns
inserted in between this~$B$ and~$T_j$. Therefore, for every instance
of~$B$ in~$M$, there must be some~$T_j$ exactly $4g$ rows above and
$4g$ columns to the right. In particular, this means that if~$M^{i,j}=B$,
then~$M^{i-4,j+4}=T_j$ for some~$j$.

The~$T_1$ above the~$P$ must be $4g$ rows above and $4g$ columns
to the right of the~$B$ to the left of the~$P$. This is only possible
if this~$B$ is in location~$M^{5,1}$, and the~$T_1$ is in location~$M^{1,5}$.

If~$M^{i,j}=B$ and~$i\neq n$, then there must another instance of~$B$ exactly
$3g$ columns to the right, and some~$T_k$ exactly $g$ rows above,
since every matrix in~$\mathcal{W}^\prime_1$, $\mathcal{W}^\prime_2$,
$\mathcal{W}^\prime_3$ or $\mathcal{W}^\prime_4$ has this pattern.
We know that these submatrices cannot be further away because
they are squeezed between the~$B$ and the~$T_j$. The new~$T_k$ must
be $4g$ rows below and $4g$ columns to the right of the new instance
of~$B$. This can only be achieved by letting~$M^{i+3,j+3}=B$
and~$M^{i-1,j+7}=T_k$.

Thus, the instances of~$B$ in a matrix~$M$ in~$D^\prime_{n}$
are exactly the matrices~${M^{3i+2,3i-2}}$, for~$1\leq i\leq n$.
Similarly, we have $M^{3i-2,3i+2}=T_j$ for some~$j$. Let~$\gamma_i$ denote
the vertex in~$\Gamma$ corresponding to the matrix~$M^{3i-2,3i+2}$.
The restrictions in~$\mathcal{F}_4^\prime$ about adjacent~$v_i$
ensure that $\gamma=\gamma_1\ldots\gamma_n$ is a path from $v_1$ to~$v_2$.

When we specify this path~$\gamma$, we uniquely define all
of the submatrices~$M^{i,j}$ except for those of the form~$M^{3i,3j}$.
Further, the restrictions from~$\mathcal{F}_4^\prime$ tell us that
each~$M^{3i,3j}$ is either a matrix of zeros or equal to~$E$ or~$Z_p$.
If the matrix at location~$M^{3i,3j}$ is nonzero, then it is uniquely
determined by the restrictions from $\mathcal{F}_4^\prime$ on the rows
of~$M$.

Let~$\pi$ be the permutation such that~$\pi(i)$ is the unique~$j$,
such that $M^{3i,3j}$ is~$E$ or~$Z_p$. Let us prove that $M=M(\gamma,\pi)$.
We already showed that $M^{2,2}=P$, $M^{3n+1,3n+1}=Q$, and
$M^{3i+2,3i-2}=B$ for all~$i$.
We know that~$M^{3i-2,3i+2}=T_j$ where~$\gamma_i=v_j$,
by the definition of~$\gamma$.

By definition, submatrix $M^{3i,3j}$ is non-zero if and only if~$\pi(i)=j$.
The restrictions on $E$ and~$Z_p$ from $\mathcal{F}_4^\prime$ imply
that $M^{3i,3j}=E$ for ${\ella(\gamma_i)=\varepsilon}$, $M^{3i,3j}=Z_p$
for~$\ella(\gamma_i)=s(Z_p)$, and
$M^{3i,3j}=Z_p$ for ${\ella(\gamma_i)=(s(Z_p))^{-1}}$.
We have a total of~$\kK$ entries~$1's$,
so all remaining~$M^{i,j}$ are zeros matrices.
This implies that $M=M(\gamma,\pi)$.\qed

\bigskip

\subsection{Proof of the \textbf{\textit{only if}} part of Lemma~\ref{TL3}.}
First, we show that if~$M(\gamma,\pi)\in \cd_n^\prime$, then~$\gamma$ is
balanced and~${\pi=\pi_\gamma}$.
By Proposition \ref{prop}, it suffices to show that the permutation~$\pi$ satisfies
the following four properties:
\begin{enumerate}
\item~$\ella(\gamma_{i})=\varepsilon$ for all~$\pi(i)=i$ ,
\item~$\ella(\gamma_{i})\in X\cup Y$ and~$\ella(\gamma_{\pi(i)})=\ella(\gamma_{i})^{-1}$, for all~$\pi(i)>i$,
\item there are no~$i$ and~$j$ with~$\ella(\gamma_{i})\sim\ella(\gamma_j)$ such that~$i<j<\pi(i)<\pi(j)$, and
\item permutation~$\pi$ is an involution.
\end{enumerate}

\smallskip

\noindent{\it Proof of {\rm(1):}} \. Consider the case where~$\pi(i)=i$.
The submatrix~$M^{3i,3i}$ is non-zero then. In this case, the $B$ at
$M^{3i+2,3j-2}$ must be the marked submatrix~$\bbs$ in some instance
of a matrix in~$\mathcal{W}_3^\prime$, $\mathcal{W}_4^\prime$
or~$\mathcal{W}_5^\prime$. Thus, $M^{3i,3i}=E$ and~$M^{3i-2,3j-2}=T_j$,
where~$\ella(\gamma_i)=\ella(v_j)=\varepsilon$, so~$\pi$ satisfies~(1).

\smallskip

\noindent{\it Proof of {\rm(2):}} \. Consider the case where~$\pi(i)>i$,
so $M^{3i,3\pi(i)}$ is non-zero. In this case, the submatrix $B$ at
$M^{3i+2,3i-2}$ must be~$\bbs$ in some instance
of a matrix in~$\mathcal{W}_1^\prime$. Thus, $M^{3i,3\pi(i)}=Z_p$,
and $M^{3i-2,3i+2}=T_j$, where~$\ella(\gamma_i)=\ella(v_j)=s(Z_p)$.
Similarly,  the submatrix~$B$ at $M^{3\pi(i)+2,3\pi(i)-2}$ must be
a marked submatrix in some matrix in~$\mathcal{W}_2^\prime$.
Thus, $M^{3\pi(i)-2,3\pi(i)+2}=T_k$, where
$\ella(\gamma_{\pi(i)})=\ella(v_k)=(s(Z_p))^{-1}$.
We conclude that $\ella(\gamma_{i})\in X\cup Y$ and
$\ella(\gamma_{\pi(i)})=\ella(\gamma_{i})^{-1}$, so~$\pi$ satisfies~(2).
\. Similar analysis shows that if~$\pi(i)<i$, then
$\ella(\gamma_{\pi(i)})\in X\cup Y$ and
$\ella(\gamma_{i})=\ella(\gamma_{\pi(i)})^{-1}$,
so $\pi^{-1}$ satisfies~(2) as well.

\smallskip

\noindent{\it Proof of {\rm(3):}} \. Assume that there exist $i$ and~$j$
with $\ella(\gamma_{i})\sim\ella(\gamma_j)$ such that~${i<j<\pi(i)<\pi(j)}$.
Let ${s(Z_p)=\ella(\gamma_i)}$ and ${s(Z_q)=\ella(\gamma_j)}$.
We then have $M^{3i,3\pi(i)}=Z_p$ and ${M^{3j,3\pi(j)}=Z_q}$.
Observe that ${M^{3j+2,3j-2}=B}$.  Since $3i<3j<3j+2$ and
$3j-2<3\pi(i)<3\pi(j)$, these three matrices together
form a pattern which is in~$\mathcal{F}_5$, a contradiction.
This implies that~$\pi$ satisfies~(3).
\. Similar analysis shows that there are no~$i$ and~$j$
with $\ella(\gamma_{i})\sim\ella(\gamma_j)$, such that~$\pi(i)<\pi(j)<i<j$,
so~$\pi^{-1}$ satisfies~(3) as well.

\smallskip

\noindent{\it Proof of {\rm(4):}} \. Note that the conditions on~$\pi$
only mention values~$i$ for which~$\pi(i)\geq i$. Therefore, even
if~$\pi$ were not an involution, the permutation~$\pi$ would have to agree
with $\pi_\gamma$ at all~$i$ such that~$\pi(i)\geq i$, by the
uniqueness of the involution~$\pi_\gamma$. However, a similar analysis
shows that $\pi^{-1}$ must also satisfy conditions (1), (2) and~(3),
so $\pi^{-1}$ also agrees with $\pi_\gamma$ at all~$i$ such
that~$\pi^{-1}(i)\geq i$.  Combining these two observations, we
conclude that $\pi=\pi^{-1}=\pi_\gamma$.\qed

\bigskip

\subsection{Proof of the \textbf{\textit{if}} part of Lemma~\ref{TL3}.}
It remains to prove that $M(\gamma,\pi_\gamma) \in \cd_n^\prime$.
Clearly, matrix~$B$ is a submatrix of $M(\gamma,\xi\pi_\gamma)$.
In addition, matrix $M(\gamma,\pi_\gamma)$ has exactly~$n$
submatrices~$B$, and each of them is the marked submatrix~$\bbs$
in some matrix in~$\mathcal{F}_4^\prime$.

Recall that there is a unique $P$ and $Q$ in~$M(\gamma,\pi_\gamma)$, and they are
sufficiently close to the edge to ensure that $M(\gamma,\pi_\gamma)$
avoids~$\mathcal{F}_2$. Similarly, we know the locations of all submatrices
$B$ and~$T_i$, and they ensure that $M(\gamma,\pi_\gamma)$
avoids~$\mathcal{F}_3$. Clearly, matrix $M(\gamma,\pi_\gamma)$ avoids~$B^\prime$,
and therefore also avoids~$\mathcal{F}_4$.
Also, if $M(\gamma,\pi_\gamma)$ did not avoid~$\mathcal{F}_5$, this would
contradict the fact that there are no~$i$ and~$j$ with
$\ella(\gamma_{i})\sim\ella(\gamma_j)$ such that
$i<j<\pi_\gamma(i)<\pi_\gamma(j)$.

Finally, we need to show that $M(\gamma,\pi_\gamma)$ avoids~$\mathcal{F}_1$.
Assume that $M(\gamma,\pi_\gamma)$ does not avoid $\mathcal{F}_1$,
and that $A$ is a $g\times g$ submatrix of~$M(\gamma,\pi_\gamma)$,
such that $A\in \mathcal{A}_g$ and $A$ is not a consecutive $g\times g$
block in~$M(\gamma,\pi_\gamma)$.
Since~$A$ is simple, every~$1$ in~$A$ comes from a different
block~$M^{i,j}$.  By definition of~$\ca_g$, matrix~$L$
is a submatrix of~$A$ (see $\S$\ref{ss:construction-notation}).
Assume the~$1$ in the center of~$L$ is above or on the main diagonal of $M$. It must therefore be in an $E$, $T_j$, or $Z_p$. The top three $1$'s
in $L$ must be in matrices of the form $Z_q$, since there is
not enough room for any of them to be in a $T_k$.

Since the relation~``$\sim$'' partitions the set of $Z_p$'s
into two equivalence classes, two of these $1$'s must be
in some~$Z_p$ and~$Z_q$ with~$Z_p\sim Z_q$.
However, there is a submatrix~$B$ that is below and to the left
of both~$Z_p$ and~$Z_q$. Together they form a matrix in~$\mathcal{F}_5$,
contradicting the fact that $M(\gamma,\pi_\gamma)$ was shown above to avoid~$\mathcal{F}_5$.
Similar analysis gives a contradiction if the~$1$ in the center
of~$L$ is below the main diagonal.

We conclude that $M(\gamma,\pi_\gamma)$ satisfies all four conditions
from Lemma~\ref{TL1}.  This implies ${M(\gamma,\pi_\gamma)\in \cd_n^\prime}$,
as desired.\qed

\bigskip

\section{Decidibility}\label{s:undecide}

\subsection{Simulating Turing Machines}
It is well known and easy to see that two-stack automata can
simulate (nondeterministic) \emph{Turing machines}.
This simulation works by using the two-stacks to represent the tape
of the Turing machine (see e.g.~\cite{HMU}). To recap, let one
stack contain everything written on the tape to the left of the head,
while the other stack contains everything written to the right.
Moving the head of the Turing machine left or right corresponds
to removing a symbol from one stack and writing a
(possibly different) symbol on the other stack.  This simulation
is direct and can be done in polynomial time.  We refer
to \cite{HMU,Sip} for the definitions and details.

\subsection{Proof of Theorem~\ref{t:undecide}}\label{ss:undecide-proof}
Let $M$ be an arbitrary deterministic Turing machine, with no input.
Consider the two-stack automaton which simulates~$M$. If $M$ does not halt,
$G(\Gamma,n)=0$ for all~$n$. If $M$ eventually halts,
then $\Gamma$ will have a single balanced path, and $G(\Gamma,n)=1$
for some~$n$.

Construct the $F$ and $F^\prime$ associated with this $\Gamma$ as
in Lemma~\ref{l:mainL}. By analogy, we then have \.
$C_n(\mathcal{F})=C_n(\mathcal{F^\prime})$~mod~$2$ \. for all \.
$n=d$~mod~$c$ \. \emph{if and only if} \. $M$ \ts does not halt.
Recall also that the \emph{halting problem} is undecidable
(see e.g.~\cite{Sip}).

Therefore, by the argument above, the construction in the proof of
Lemma~\ref{l:mainL} can be emulated by a Turing machine,
and the construction satisfies \.
$C_n(\mathcal{F})=C_n(\mathcal{F^\prime})$~mod~$2$ \. for all \.
$n\neq d$~mod~$c$. Thus, given~$M$, we can construct finite sets
$\cf$ and $\cf^\prime$ permutation patterns,
such that $C_n(\mathcal{F})=C_n(\mathcal{F^\prime})$ for all $n\ge 1$
if and only if $M$ does not halt. Therefore, it is undecidable whether
$C_n(\mathcal{F})=C_n(\mathcal{F^\prime})$ for all $n\ge 1$.\qed

\subsection{Implications} \label{ss:fin-rem-undecide}
Theorem~\ref{t:undecide} has rather interesting implications for
our understanding of pattern avoidance (cf.~$\S$\ref{ss:fin-rem-P-recursive}).
For example, a standard argument (see e.g.~\cite[p.~2]{Poon}),
 gives the following surprising result:

\begin{cor}\label{cor:undecide-ZFC}
There exist two finite sets of patterns $\cf$ and~$\cf'$, such that the
problem of whether \.
$C_n(\cf) = C_n(\cf') $ \. {\rm mod\.{}2} \. for all \ts $n\in \nn$, \ts
is independent of~{\rm \ts ZFC}.
\end{cor}

We would be curious to see an explicit bound on the size of such sets
of patterns, but it is likely to be quite large.
Here is another immediate application of the theorem:

\begin{cor}\label{cor:undecide-first-change}
For all $k$ large enough, there exist two finite sets of patterns
$\cf$ and~$\cf'$, such that the smallest $n$ for which \.
$C_n(\cf) \ne C_n(\cf')$ \. {\rm mod~2}, \ts satisfies \ts
{\rm $n> \rA(n)$}.
\end{cor}

Here $\rA(n)$ is the Ackermann function (see e.g.~\cite{AF}),
but any other computable function can be used in its place.

\subsection{Speculations} \label{ss:fin-rem-undecide-conj}
These results give us confidence in the following open prolems,
further extending Theorem~\ref{t:undecide}.

\begin{conj}[Parity problem] \label{conj:undecide-even}
The problem of whether \. $C_n(\cf) = 0$ \. {\rm mod~2} \. for all \ts $n\in \nn$, \ts
is~undecidable.
\end{conj}

We believe that our tools will prove useful to establish this conjecture,
likely with a great deal more effort. This goes beyond the scope of the paper.

\begin{conj}[Wilf--equivalence problem]
\label{conj:undecide-equal}
The problem of whether \. $C_n(\cf_1) = C_n(\cf_2)$ \.
for all \ts $n\in \nn$ \ts is undecidable.
\end{conj}

This conjecture is partly motivated by the recently found
unusual examples of Wilf--equivalence~\cite{BP}.  It is
more speculative than Conjecture~\ref{conj:undecide-even},
since in our construction $C_n(\cf)-C_n(\cf')$
grows rather rapidly.  It is thus conceivable that in contrast with
Corollary~\ref{cor:undecide-first-change},
the smallest~$n$ for which $C_n(\cf)\ne C_n(\cf')$
is constructible, perhaps even quite small in the
size of the patterns.  This would imply that the mod-2
problem is computationally harder than the Wilf--equivalence
problem, a surprising conclusion.

In a different direction, it would be also very interesting
to resolve the parity and the Wilf--equivalence problems
when only one permutation is avoided.   In this case, we believe that
both are likely to be decidable. 

\begin{conj}
\label{conj:undecide-single}
For forbidden sets with a single permutation $|\cf|=|\cf'|=1$, both
the parity and the Wilf-equivalence problems are decidable.
\end{conj}

In order to clarify and contrast the conjectures, let us note that
non-P-recursiveness is quite possible and rather likely for
permutation class avoiding one or two permutations
(see~$\S$\ref{ss:fin-rem-1324}).  However, the undecidability is
a much stronger condition and we do not believe one permutation
has enough room to embed all logical axioms.  This
``small size is hard to achieve'' phenomenon is somewhat
similar to aperiodicity and undecidability for plane tilings
(see~$\S$\ref{ss:fin-rem-tilings}).

\bigskip

\section{Wilfian formulas}\label{s:wilfian}

\subsection{Complexity preliminaries}\label{ss:wilfian-complexity}
Recall the $\ParP$ complexity class which is a parity version of
the class of counting problem~$\SP$, see e.g.~\cite{Pap}. For example,
\textsf{$\oplus$2SAT} and the parity of the number of Hamiltonian cycles
are $\ParP$-complete, see~\cite{Val}. It is strongly believed that
$\ParP\ne \P$.  In fact, if $P = \ParP$, then $\PH = \BPP$ by Toda's theorem.
In particular, this implies $\NP = \BPP$, contradicting widely
held beliefs that $\P = \BPP$ and $\P \ne \NP$.

Recall the exponential complexity classes~$\EXP$, $\NEXP$ and its
counting counterpart~$\SEXP$.
Let us now formally define an exponential time version of~$\ParP$.
A language $D$ is in $\PAR$ if there exists a nondeterministic
Turing machine $M$ such that $M$ has an \emph{odd number} of accepting
paths on input $n$ if and only if $n\in D$, and every path in
$M$ halts in time $O(2^{p(\ell)})$ where $p$ is a polynomial,
and $\ell$ is the length of the input in bits. For example,
computing parity of the number of non-isomorphic Hamiltonian
graphs on $n$ vertices is in~$\PAR$, see~\cite[A003216]{OEIS}.
Similarly, for a fixed~$\cf$, computing $C_n(\cf)$\,mod~2 is in~$\PAR$
by definition (cf.~\ref{ss:fin-rem-wilfian}).

Now,  it is believed that $\EXP \ne \PAR$ for reasons similar
to $\P \ne \ParP$.  First, note that $\P = \ParP$ implies $\EXP = \PAR$,
see e.g.~\cite{BBF}.  In the opposite direction, it follows from the
proof of the main result in~\cite{HIS}, that
$\EXP = \PAR$ implies that there are no \emph{polynomially sparse languages}
in $\ParP\backslash\P$ (see~\cite{Pap} for definitions).
This is an unlikely conclusion given how far apart $\P$ and $\ParP$
are.\footnote{This argument and further evidence against $\ts \EXP = \PAR\ts$ was communicated to us
by Joshua~Grochow on \emph{CS Theory Stack Exchange}: \ts {\tt http://tinyurl.com/m79o4dv}.}
For more discussion of~$\PAR$ in relation to other complexity classes, see~\cite{BBF}
(see also~$\S$\ref{ss:fin-rem-complexity}).

\subsection{Proof of Theorem~\ref{t:wilfian}}\label{ss:wilfian-proof}
Let $D$ be a language in~$\PAR$. There exists a nondeterministic
Turing machine $M$ and a polynomial $p$ such that $M$ has an odd
number of accepted paths on input $n$ if and only if $n\in D$,
and every path in $M$ halts in time at most $2^{p(\log n)}$ on
input~$n$.

Let us form $M^\prime$ modifying~$M$, by artificially extending
every accepted path in such a way that it instead accepts at exactly time
$\lfloor 2^{p^\prime(\log n)}+n\rfloor$ for some $p^\prime\geq p$.
We further assume that $p^\prime$ is increasing, so
$\lfloor 2^{p^\prime(\log n)}+n\rfloor$ is injective on~$\mathbb{N}$.
Therefore, all accepted paths of $M^\prime$ of length
$\lfloor 2^{p^\prime(\log n)}+n\rfloor$ must come from input~$n$.

Let us form $M^{\prime\prime}$ by modifying $M^\prime$
in such a way that instead of taking an input $n$,
it nondeterministically chooses a natural number~$n$.
The number of accepted paths of length
$\lfloor 2^{p^\prime(\log n)}+n\rfloor$
will still be odd if and only if $n\in D$.
We can now build a two-stack automaton $\Gamma$ which simulates
$M^{\prime\prime}$ so that $G(\Gamma,\lfloor 2^{p^{\prime\prime}(\log n)}+n\rfloor)$
is equal to the number of accepted paths in $M$ of length
$\lfloor 2^{p^\prime(\log n)}+n\rfloor$.

Finally, we construct the sets of patterns $\cf$ and $\cf^\prime$
associated with this~$\Gamma$, following the construction in
Lemma~\ref{l:mainL}. We have then:
$$
G(\Gamma,\lfloor 2^{p^{\prime\prime}(\log n)}+n\rfloor)
\, = \, C_m(\mathcal{F})\. - \. C_m(\mathcal{F}^\prime)\mod 2\ts,
$$
where $m=c\lfloor 2^{p^{\prime\prime}(\log n)}+n\rfloor+d$.
This implies that $n\in D$ if and only if \ts $C_m(\mathcal{F})-C_m(\mathcal{F}^\prime)$
is odd.

Assume now that $\mathcal{F}$ and $\mathcal{F}$ have a Wilfian formula.
Then we can compute ${C_m(\mathcal{F})-C_m(\mathcal{F}^\prime)}$
and thus decide if $n\in D$ in time polynomial in $m$.
Clearly, time polynomial in~$m$ is time exponential in~$(\log n)$.
This implies that $\EXP= \PAR$, a contradiction.\qed 

\bigskip

\section{Final remarks and open problems}\label{s:fin-rem}

\subsection{}\label{ss:fin-rem-growth}
Despite a large literature on pattern avoidance, relatively
little is known about general sets of patterns.  Notably,
the \emph{Stanley--Wilf conjecture} proved by Marcus and Tardos
shows that $C_n(\cf)$ are at most exponential~\cite{MT}, improving
on an earlier near-exponential bound by Alon and Friedgut~\cite{AF}.
Most recently, Fox showed that for a fixed~$k$ and a random 
permutation $\om \in S_k$, we have
$C_n(\cf) = \exp\bigl(k^{\Theta(1)}n\bigr)$~\cite{Fox}
(see also~\cite[$\S$2.5]{Vat3}).

\subsection{}\label{ss:fin-rem-asy}
For an integer P-recursive sequence $\{a_n\}$, the generating series
$\ts \ca(t) = \sum_{n=0}^\infty \. a_n \ts t^n$ \ts
is \emph{D-finite} (\emph{holonomic}), i.e.~satisfies a linear ODE (see e.g.~\cite{FS,Sta}).
In the case $a_n = e^{O(n)}$, the series $\ca(t)$ is also a \emph{$G$-function},
an important notion in Analytic Number Theory (see e.g.~\cite{Gar}).

\subsection{}\label{ss:fin-rem-equiv}
The most celebrated example of Wilf-equivalence is $(123)\sim (213)$.
It follows from results of MacMahon~(1915) and
Knuth~(1973), that \ts $C_n(123) =C_n(213) = \frac{1}{n+1}\binom{2n}{n}$,
the \emph{$n$-th Catalan number}, see e.g.~\cite{Kit,Sta}.  Many other
Wilf-equivalent classes are known now, see e.g.~\cite{Kit}.

\subsection{}\label{ss:fin-rem-1324}
There are three Wilf-equivalence classes of patterns in~$S_4$.  Two of them
are known to give P-recursive sequences, but whether $\{C_n(1324)\}$ is
P-recursive remains a long-standing open problem in the area
(see e.g.~\cite{Ste,Vat3}). The problem proved so challenging,
Zeilberger himself seems to have abandoned Conjecture~\ref{conj:NZT}
because of it, see~\cite{EV}.

We should mention that there seems to be some recent strong experimental
evidence \emph{against} sequence $\{C_n(1324)\}$ being P-recursive.  Numerical
analysis of the values for $n\le 36$ given in~\cite{CG} 
suggest the asymptotic behavior
$$
C_n(1324) \sim A \cdot \la^n \cdot \mu^{\sqrt{n}}\cdot n^\al\,,
$$
where $A \approx 8$, $\la \approx 11.6$, $\mu \approx 0.04$, and $\al \approx -1.1$.
If this asymptotics holds, this would imply non-P-recursiveness by Theorem~7
in~\cite{GP1}, which forbids $\mu^{\sqrt{n}}$ terms.

For larger sets, the leading candidates for non-P-recursiveness are the
sequences $\{C_n(4231, 4123)\}$ and $\{C_n(4123, 4231, 4312)\}$ which
were recently computed for $n\le 1000$ and $5000$, respectively~\cite{AH+}.

\subsection{}\label{ss:fin-rem-1324-parity}
There is a bit of a controversy over a question whether the integer
$N:=C_{1000}(1324)$
is humanly or supernaturally computable, with Steingr{\'{\i}}msson and
Zeilberger occupying the opposite sides of the argument, see~\cite{EV,Ste}.
From a formal Theoretical~CS point of view, this is really an argument
over the existence of a Wilfian formula to compute $\{C_{n}(1324)\}$ (cf.~\cite{MR}).
This raises a curious question whether computing $\{C_{n}(1324)\ \text{mod}\,2\}$
is in~$\EXP$, which is perhaps easier to resolve.  Leaving aside
philosophical and theological issues, the proof of Theorem~\ref{t:wilfian}
suggests a surprising possibility that even the \emph{parity} of $N$ might
be hard to compute.

Note also that $C_{n}(1324)=V_{n}(1324)\ \text{mod}\,2$, where $V_{n}(\om)$
is the number of \emph{involutions} $\si \in S_n$ avoiding~$\om$.  The sequence
$\{V_{n}(1324)\}$ is still difficult to analyze (see~\cite{BHPV}). It is known
to have a much smaller growth, so perhaps easier to compute.

\subsection{} \label{ss:fin-rem-size-small}
The bound $|\cf| < 10^{119}$ in Corollary~\ref{c:count}
can easily be improved down to a bit under 30,000 by refining the
proof of Theorem~\ref{t:count} and using small technical tricks.
While it is possible that a single permutation suffices (see above),
it is unlikely that the tools from this paper can be used. \ts

\subsection{} \label{ss:fin-rem-tilings}
It is worth comparing the non-P-recursiveness of pattern avoidance
with the history of \emph{aperiodic tilings} in the plane.  Originally
conjectured by Wang to be impossible~\cite{Wang}, such tilings were first
constructed by Berger in 1966 by an undecidability argument~\cite{Ber}.
Note that Berger's construction used about 20,000 tiles.  In 1971,
by a different undecidability argument, Robinson was able to reduce
this number to six~\cite{Rob}. Most recently, Ollinger established the
current record of five polyomino tiles~\cite{Oll}.  Without undecidability,
Penrose famously used only two tiles to construct aperiodic tilings
(see e.g.~\cite{GS}).  Whether there exists one tile which forces
aperiodicity remains open.  This is the  so-called \emph{einstein problem}
(cf.~\cite{ST} and~\cite[$\S$7.5]{Yang}).
See also~\cite{Boas,GI,Pak-horizons} for the complexity of several other
related general tiling problems.

\subsection{} \label{ss:fin-rem-complexity}
Note that we do not claim that computing the parity of
$C_n(\cf)$ is $\PAR$-complete.  This is because our argument
would have $\cf$ and~$\cf'$ as an input, which is tangential
to the Klazar--Vatter question in the introduction.  \ts
Still, it would be interesting to prove the following complexity
variation on Theorem~\ref{t:wilfian}:

\begin{conj}\label{conj:undecide-par-complete}
There exists a finite set of patterns~$\cf$, such that
computing $\.\{C_n(\cf)\}\.$ is \. {\rm $\SEXP$}-complete, and
computing
$\.\{C_n(\cf)$~{\rm mod}~$2\}\.$ is \. {\rm $\PAR$}-complete.
\end{conj}

To give an example of a $\PAR$-complete problem, one can take a
parity version of a known $\SEXP$-complete problem in cases
where the reduction is parsimonious.  For example, the parity
of $\ts \frac16 \ts$ times the number of $3$-colorings in
\emph{succinct graphs}~\cite{GW} is a natural candidate
(see~\cite{Bar} for a parsimonious reduction).\footnote{This
approach was suggested to us by Greg Kuperberg on
\emph{CS Theory Stack Exchange}: \ts {\tt http://tinyurl.com/m79o4dv}.}
Perhaps the closest in spirit to the conjecture are the $\NEXP$-complete
tiling problems considered in~\cite{GI}, where the input is a single
integer~$n$, rather than a set of tiles  (cf.~\cite{Boas,Pak-horizons}).

\subsection{}\label{ss:fin-rem-wilfian}  \emph{Wilfian formulas}
were introduced by Wilf in 1982, see~\cite{Wilf-answer}.
They have been christened and received some attention only very
recently~\cite{Vat2,Zei}.\footnote{In the original paper~\cite{Wilf-answer},
Wilf uses two notions of a \emph{formula} which led to some confusion
in the definitions.  In particular, Klazar in~\cite{Kla} uses a more
restrictive notion which is equivalent to the one we use for
exponentially growing sequences.}
We believe this is the first paper which examines them from the
computational complexity point of view.  Of course,
pattern avoidance itself is a special case of the
\emph{pattern matching problem}, with a large algorithmic
and complexity literature.  See e.g.~\cite{GM} for the recent
breakthrough and~\cite[$\S$8.2.1]{Kit} for further references.

We should mention that all P-recursive sequences trivially
have Wilfian formulas, but the opposite is false.  For example,
the sequence of primes is not P-recursive~\cite{FGS}, but the $n$-th
prime is clearly computable in exponential time.  Same with the
number $p(n)$ of partitions of~$n$.  The asymptotics \ts
$p(n) \sim C \cdot n^{-1} \cdot \mu^{\sqrt{n}}$ \ts implies that
it is not P-recursive (see above), but Euler's recurrence (see e.g.~\cite{Pak}),
shows it is computable in time polynomial in~$n$.

A quick warning to the reader.  Just because the sequence
$\{a_n\}$ grows exponentially and has no known $o(a_n)$
time algorithm to compute it, does not mean that $\{a_n \mts\mts\mts\mod 2\}$
is also hard.  Typical examples mentioned in the literature are
the number of self-avoiding paths~\cite{Zei}
(see also~\cite[A001411]{OEIS}), and the number
of unlabeled graphs on $n$ vertices~\cite{Wilf-answer}
(see also~\cite[A000088]{OEIS}).  In both cases, the sequences
are even for $n$ large enough (the first is obvious, the second
is proved in~\cite{CR}).

\subsection{} \label{ss:fin-rem-P-recursive}
In~\cite{EV}, Zeilberger asks to characterize pattern
avoiding permutation classes (of single permutations)
with g.f.'s rational, algebraic, or D-finite.
We believe a complete characterization of the latter
class is impossible.

\begin{open}\label{op:undecide-P-recursive}
The problem of whether \. $\{C_n(\cf)\}$ \. is \. \emph{P-recursive} \ts
is undecidable.
\end{open}

Let us note that our results offer only a weak evidence in favor
of the open problem, since Theorem~\ref{l:mod2} gives only a
necessary condition on P-recursiveness, and  (cf.~\cite{GP1}).
It would be interesting to see if any of Zeilberger's problems
are decidable. A rare decidability result in this direction is given
in~\cite{BRV} (see also a discussion in~\cite[$\S$3]{Vat3}).

\subsection{}\label{ss:fin-rem-ADE}
The original form of the Noonan--Zeilberger Conjecture makes a
stronger claim: \ts for every $\cf=\{\om_1,\ldots,\om_\ell\}$ and an
integer vector $\br=(r_1,\ldots,r_\ell) \in \nn^\ell$, the sequence
$\{C_n(\cf,\br)\}$ is P-recursive, where $C_n(\cf,\br)$
is the number of permutations $\si \in S_n$ which contain
exactly $r_i$ copies of pattern~$\om_i$, for all $1\le i \le \ell$.
It was shown by Atkinson
in~1999, that this stronger claim is equivalent
to the \ts $r_1=\ldots = r_\ell = 0$ \ts case,
which is our Conjecture~\ref{conj:NZT}, see~\cite{Atk}.

Other variations on the Noonan--Zeilberger Conjecture have
also been studied in the literature, including the case
of consecutive patterns~\cite{Eli}, set partitions~\cite{Sag}, and patterns with
infinite support  (see~\cite{Vat3}).
Theorem~\ref{t:NZT} can perhaps be further extended to show
that for some set~$\cf$ the g.f.~for $\{C_n(\cf)\}$ is
not ADE (see e.g.~\cite{Sta}).  We plan to explore this
problem in~\cite{GP3}.

\subsection{}\label{ss:fin-rem-GP}
Some basic ideas in this paper have been first developed in
a much easier case of Wang tilings of the square, studied
in~\cite{GP2} (see also~\cite[$\S$9.8.2]{Kit}).  Curiously,
a key tool is Theorem~\ref{l:mod2}, which we proved
in~\cite{GP1} by a short self-contained argument.  Note
that this lemma is used in~\cite{GP1} also to prove the non-P-recursiveness
of the probabilities of returns in linear groups, but is applied
in a very different way.

\vskip.5cm

\nin
{\bf Acknowledgements.} \ts  We are very grateful to Michael Albert,
Matthias Aschenbrenner, Milk\'{o}s B\'{o}na, Alexander Burstein,
Ted Dokos, Lance Fortnow, Joanna Garrabrant, Ira Gessel, Joshua Grochow,
Greg Kuperberg,
Andrew Marks, Sam Miner, Alejandro Morales, Greta Panova, Andrew Soffer,
Richard Stanley, Vince Vatter and Jed Yang for helpful conversations
and useful comments at different stages of the project.
Special thanks to Doron Zeilberger for telling us about importance
of the conjecture and his ambivalence about its validity.
%
The first author was partially supported by the University of California
Eugene V.~Cota-Robles Fellowship; the second author was partially
supported by the~NSF.

 \newpage


\end{document}